\documentclass[11pt,a4paper,reqno]{amsart}%
\usepackage[a4paper,left=2.0cm,right=2.0cm,top=2.5cm,bottom=2.5cm]{geometry} 
\newtheorem{theorem}{Theorem}[section] 
\newtheorem{lemma}[theorem]{Lemma}

\newtheorem{proposition}[theorem]{Proposition} 
\newtheorem{remark}[theorem]{Remark} 
\usepackage{mathrsfs,amssymb,amsmath,amsthm,stmaryrd,graphicx,mathtools,xfrac,cite}
\usepackage[amsstyle]{cases}
\usepackage{enumitem}
\usepackage[linkcolor=blue,citecolor=cyan,colorlinks]{hyperref}
\usepackage[utf8]{inputenc}
\usepackage[T1]{fontenc}
\usepackage{xcolor,ulem}   
\numberwithin{equation}{section}
\allowdisplaybreaks
\renewcommand{\ge}{\geqslant}
\renewcommand{\le}{\leqslant}
\usepackage[integrals]{wasysym}
\renewcommand{\int}{\varint}
\renewcommand{\to}{\rightarrow} 
\newcommand{\divergence}{\nabla\cdot }  
\newcommand{\vorticity}{\nabla\!\times\! }    
\newcommand{\Bdnabla}{\overline{\nabla} } 
\newcommand{\Bddiv}{\Bdnabla\cdot }  
\newcommand{\DT}{\mathfrak{D}_t }  
\newcommand{\energy}{\mathscr{E}}  
\newcommand{\mc}{\mathscr{H}} 
\newcommand{\sff}{\mathrm{II}} 
\newcommand{\LB}{\Delta_{\sff}} 
\newcommand{\parOmega}{\partial\Omega}
\usepackage{upgreek}
\newcommand{\normal}{n}  
\newcommand{\radi}{\mathscr{R}} 
\newcommand{\norm}[1]{\left\| #1 \right\|} 
\newcommand{\module}[1]{\left| #1 \right|} 
\newcommand{\Brace}[1]{\left\{ #1 \right\}} 
\newcommand{\Paren}[1]{\left( #1 \right)} 
\newcommand{\Bracket}[1]{\left[ #1 \right]} 
\newcommand{\CC}{\mathscr{C}_\dagger}
\newcommand{\Tstar}{T^\dagger}
\newcommand{\error}{\mathfrak{R}} 
\newcommand{\tencon}{\boldsymbol{\ast}}

\usepackage{epstopdf}
\begin{document}
\title[Finite-time blow-up mechanisms]{Classification of finite-time blow-up mechanisms for the incompressible free-boundary Euler equations with surface tension}
\author{Chengchun Hao, Tao Luo, \and Siqi Yang}
\address{Chengchun Hao\newline
State Key Laboratory of Mathematical Sciences, Academy of Mathematics and Systems Science, Chinese Academy of Sciences, Beijing 100190, China \newline
\and
School of Mathematical Sciences, University of Chinese Academy of Sciences, Beijing 100049, China
}
\address{Tao Luo\newline
Department of Mathematics, City University of Hong Kong, Hong Kong, China}
\address{Siqi Yang\newline School of Mathematics, Statistics and Mechanics, Beijing University of Technology, Beijing 100124, China}
\email{hcc@amss.ac.cn}
\email{taoluo@cityu.edu.hk}
\email{siqiyang@bjut.edu.cn}
\begin{abstract}  
We establish a blow-up criterion for strong solutions of the three-dimensional incompressible Euler equations with surface tension in a bounded domain with a closed moving free boundary. The criterion is formulated at the $H^3\times H^4$ regularity level of the Shatah--Zeng local well-posedness theory and imposes \textit{no} assumptions of symmetry, periodicity, graph structure, or simple connectedness. If the maximal existence time $T<\infty$, then at least one of the following four mechanisms must occur: (i) first self-intersection of the free boundary; (ii) loss of mean curvature regularity in $H^{\frac{3}{2}}$, or loss of boundary regularity in $H^{2+\varepsilon}$ for any sufficiently small fixed $\varepsilon>0$; (iii) loss of $H^{\frac{5}{2}}$ regularity of the normal boundary velocity; or (iv) $L^1_tL^\infty$ blow-up of the interior velocity gradient. For simply connected domains, the interior alternative admits a
refinement involving only the $L^1_tL^\infty$-norm of the vorticity,
and this refinement recovers exactly the classical
Beale--Kato--Majda criterion in the fixed-boundary case.
For irrotational flows in the simply connected free-boundary setting,
the criterion reduces to the three boundary mechanisms. 
\end{abstract}  
\keywords{Free boundary problem, incompressible Euler equations, finite-time blow-up, Beale--Kato--Majda criterion, splash singularity, surface tension}
\subjclass[2020]{35Q35, 35R35, 35B44, 76B03, 76B45}
\maketitle
\tableofcontents
\section{Introduction} 
We consider the three-dimensional incompressible Euler equations with surface tension in a bounded domain $\Omega_t\subset \mathbb{R}^3$, with a closed free surface $\parOmega_t$ evolving in time (see Fig.~\ref{fig:domain}):
\begin{subnumcases}
{\label{eqEuler}}
\DT v+\nabla p=0,& \text{ in } \Omega_t,\label{eqEuler1} \\ 
\divergence v=0, & \text{ in } \Omega_t, \label{eqEuler2}\\ 
p=\mc_{\parOmega_t},\quad v_n=\mathscr{V}_{\parOmega_t},  & \text{ on }  \parOmega_t, \label{eqEuler3}\\ 
v(\cdot,0)=v_0, & \text{ in } \Omega_0,
\end{subnumcases} 
where $t>0$ denotes time, $\DT \coloneqq \partial_t + v \cdot \nabla$ is the material derivative, $v=v(x,t)$ is the velocity field, and $p=p(x,t)$ is the scalar pressure. On $\parOmega_t$, $\normal = (\normal^1, \normal^2, \normal^3)^\top$ denotes the unit outer normal, $\mc_{\parOmega_t}$ denotes the mean curvature, and $\mathscr{V}_{\parOmega_t}$ is the normal velocity of $\parOmega_t$, which equals the normal component $v_n\coloneqq v\cdot\normal$ of the fluid velocity.
The initial domain $\Omega_0$ has $H^4$-regular boundary and the initial velocity field $v_0$ lies in $H^3(\Omega_0)$. For simplicity, the surface tension coefficient is normalized to one.
\begin{figure}[h]
\centering
\includegraphics[width=0.6\linewidth]{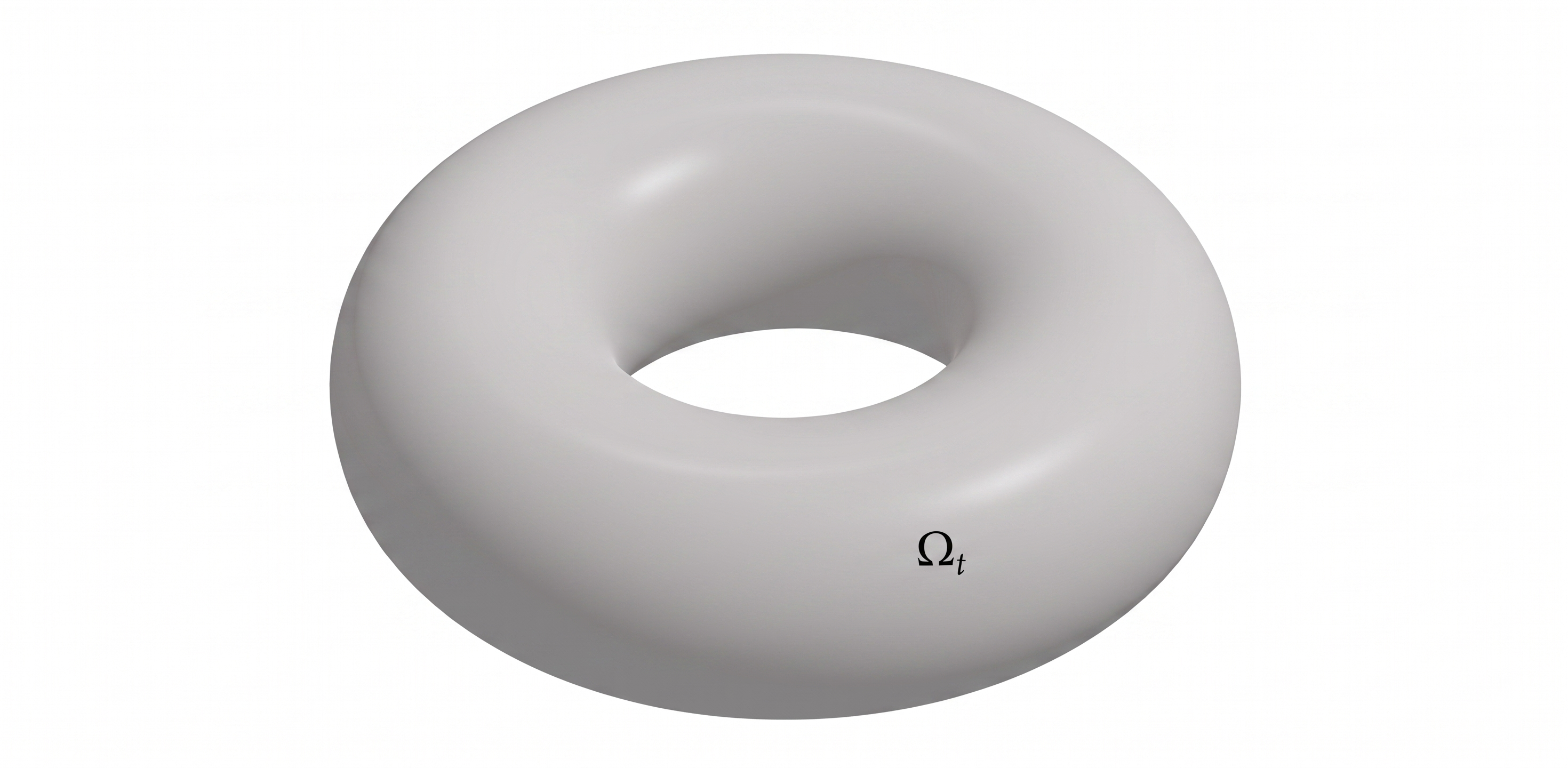}
\caption{Bounded fluid domain with a closed free surface.}
\label{fig:domain}
\end{figure} 

The free-boundary problem for the incompressible Euler equations has been studied extensively over several decades, with advances addressing irrotational and rotational flows, interfacial conditions, and singularity formation. For irrotational flows without surface tension, foundational works established well-posedness \cite{Wu1999,Lannes2005,Wu2011,Germain2012,Alazard2014,Ionescu2015,Alazard2015,Ifrim2017,Wang2018}. Surface tension was later incorporated as a regularizing mechanism, leading to well-posedness results \cite{Iguchi2001,Alazard2011,Deng2017,Ifrim2017} and rigorous zero-surface-tension limits \cite{Ambrose2005,Ambrose2009}. 
For rotational flows, Christodoulou and Lindblad \cite{Christodoulou2000} developed pioneering a priori estimates, and local well-posedness without surface tension was subsequently proved in \cite{Lindblad2005,Coutand2007,Zhang2008}. 
For the interaction of surface tension and vorticity, we refer to \cite{Ogawa2002,Schweizer2005,Coutand2007,Shatah2008a,Shatah2011,Disconzi2019}; for more recent references, see \cite{Ming2024,Aydin2024,Hu2024,Ifrim2025}.

Beyond well-posedness, regular free-boundary solutions may develop finite-time self-intersections under suitable initial conditions, leading to loss of injectivity of the flow map. These singularities appear either as pointwise \textit{splash singularities} or as higher-dimensional \textit{splat singularities} (along an arc in two dimensions or along a surface in three dimensions), as illustrated in Fig.~\ref{fig:case1}. 
A key feature is that, in a parametric representation, the surface may remain regular in parameter space up to the moment of self-intersection. Such singularities were first observed for two-dimensional water waves without surface tension \cite{Cordoba2010,Castro2013} and were later shown to persist in the presence of surface tension \cite{Castro2012b}. Subsequent studies extended these phenomena to rotational Euler equations \cite{Coutand2014} and Navier--Stokes systems \cite{Coutand2019,Castro2019,giga2026}; further developments appear in \cite{Coutand2019a,DiIorio2020a,Cordoba2021,Jeon2024,LuoYang2026}.
\begin{figure}[htb]
\centering
\includegraphics[width=0.8\linewidth]{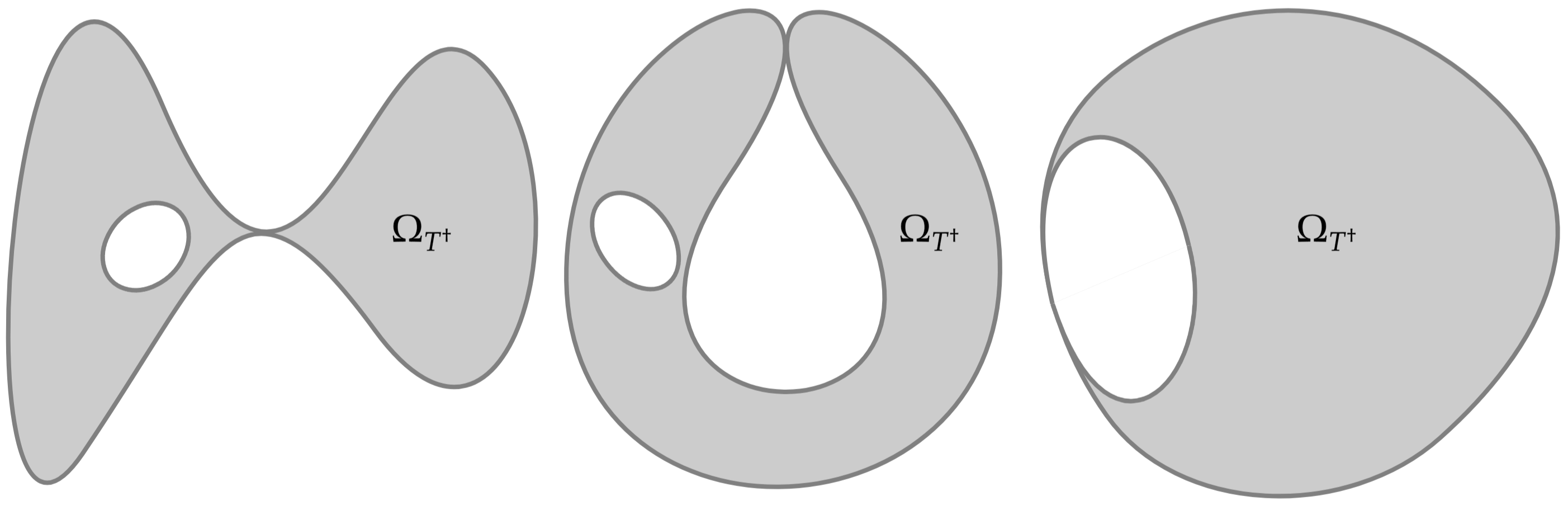}
\caption{Representative cases of self-intersection of the free boundary.}
\label{fig:case1}
\end{figure}

The formation of self-intersections is therefore distinct from loss of Sobolev regularity. Blow-up criteria describing regularity breakdown have been developed systematically for incompressible rotational Euler equations without surface tension, including the graph-based framework for bounded domains \cite{Wang2021a}, the result for initial domains diffeomorphic to a ball \cite{Ginsberg2021}, and the sharp characterization for general bounded domains without a simple-connectedness assumption \cite{Ifrim2025}. When surface tension is included in \eqref{eqEuler3}, related blow-up results have been obtained in \cite{Julin2024} (at vanishing electric fields), \cite{Luo2024}, and \cite{Hao2025} (under zero magnetic fields). 

Since the present work focuses on surface tension effects, we briefly recall several blow-up criteria from \cite{Julin2024,Luo2024,Hao2025}.
For general bounded domains, \cite{Julin2024} represents the free boundary by a height function $h$ defined over a fixed smooth reference surface $\Gamma$ and proves that finite-time singularity formation requires either
\begin{enumerate}[label={\textup{(\roman*)}},topsep=0pt,itemsep=0pt]
\item topological breakdown of the height function representation, or
\item divergence of the composite norm:
\begin{equation*}
\sup_{0\le t< \Tstar}\Paren{\norm{h(\cdot,t)}_{C^{1,\alpha}(\Gamma)}+\norm{\nabla v(\cdot,t)}_{L^\infty(\Omega_t)}+\norm{v_n(\cdot,t)}_{H^2(\parOmega_t)}}=\infty.
\end{equation*}
\end{enumerate}
Here the $L^\infty$-in-time blow-up condition
\begin{equation*}
\sup_{0\le t< \Tstar}\norm{\nabla v(\cdot,t)}_{L^\infty(\Omega_t)}=\infty
\end{equation*}
is stronger in time than its integral counterpart
\begin{equation*}
\int_{0}^{\Tstar}\norm{\nabla v(\cdot,t)}_{L^\infty(\Omega_t)}dt=\infty,
\end{equation*} 
since boundedness of the supremum norm implies boundedness of the $L^1$ norm when $\Tstar\in (0,\infty)$.

In \cite{Luo2024}, a blow-up criterion for solutions $(v,\Omega_t)$ of the free-boundary Euler equations was established in the periodic simply connected graph domain
\begin{equation*}
\Omega_t=\Brace{(x_1,x_2,x_3)\in\mathbb{R}^3:(x_1,x_2)\in\mathbb{T}^2,-b<x_3\le \psi(t,x_1,x_2)},
\end{equation*}
with regularity  $\Paren{v(t),\psi(t)}\in H^s(\Omega_t)\times H^{s+1}(\mathbb{T}^2)$ for $s>\frac{9}{2}$. If the maximal existence time $\Tstar<\infty$, then at least one of the following scenarios must occur:
\begin{align}
\int_{0}^{\Tstar}\norm{(\partial^\varphi\!\times\! v)(t)}_{L^{\infty}}dt&=\infty,\label{eqluo2024thm1.7b}\\
\limsup_{t\to\Tstar}\Paren{\norm{\psi(t)}_{C^3}+\norm{\partial_t\psi(t)}_{C^3}+\norm{\partial^2_{tt}\psi(t)}_{H^{3/2}}}&=\infty,\label{eqluo2024thm1.7a1}\\
\limsup_{t\to\Tstar}\Paren{\int_0^t\norm{(v_1,v_2)(s)}_{\dot{W}^{1,\infty}}ds+\norm{(v_1,v_2)(t)}_{L^{\infty}}}&=\infty,\label{eqluo2024thm1.7a2}\\
\limsup_{t\to \Tstar} \Paren{\frac{1}{\partial_3\varphi(t)}+\frac{1}{b-\norm{\psi(t)}_{L^\infty}}}&=\infty,\ \text{or turning occurs on}\ \parOmega_{\Tstar},
\label{eqluo2024thm1.7c}
\end{align} 
where $\varphi(t,x_1,x_2,x_3)=x_3+\chi(x_3)\psi(t,x_1,x_2)$ with a cut-off function $\chi\in C_0^\infty(-b,0]$ (cf. \cite[(1.5)]{Luo2024}), $\partial_a^\varphi=\partial_a-\frac{\partial_a \varphi}{\partial_3\varphi}\partial_3$ for $a=1,2$, and $\partial_3^\varphi=\frac{1}{\partial_3\varphi}\partial_3$. On the free boundary, \eqref{eqluo2024thm1.7a1} and \eqref{eqluo2024thm1.7c} involve rather complicated quantities, including the second-order time derivative of the graph function.

Moreover, the graph-based criteria in \cite{Luo2024} (and in \cite{Wang2021a} for the case without surface tension) have an inherent geometric limitation: they rely on a global single-valued height function $\psi$. Such a representation ceases to be applicable when the free boundary undergoes \textit{turning} or \textit{self-intersection}, since these scenarios violate the monotone dependence on the vertical variable $x_3$. Consequently, graph formulations naturally describe \textit{top-bottom contact singularities} ($\psi \to -b$), but they do not capture more general boundary evolution such as folding or pinch-off, which requires non-graph representations; see Fig.~\ref{fig:turnandcontact}.
\begin{figure}[htb]
\centering
\includegraphics[width=0.8\linewidth]{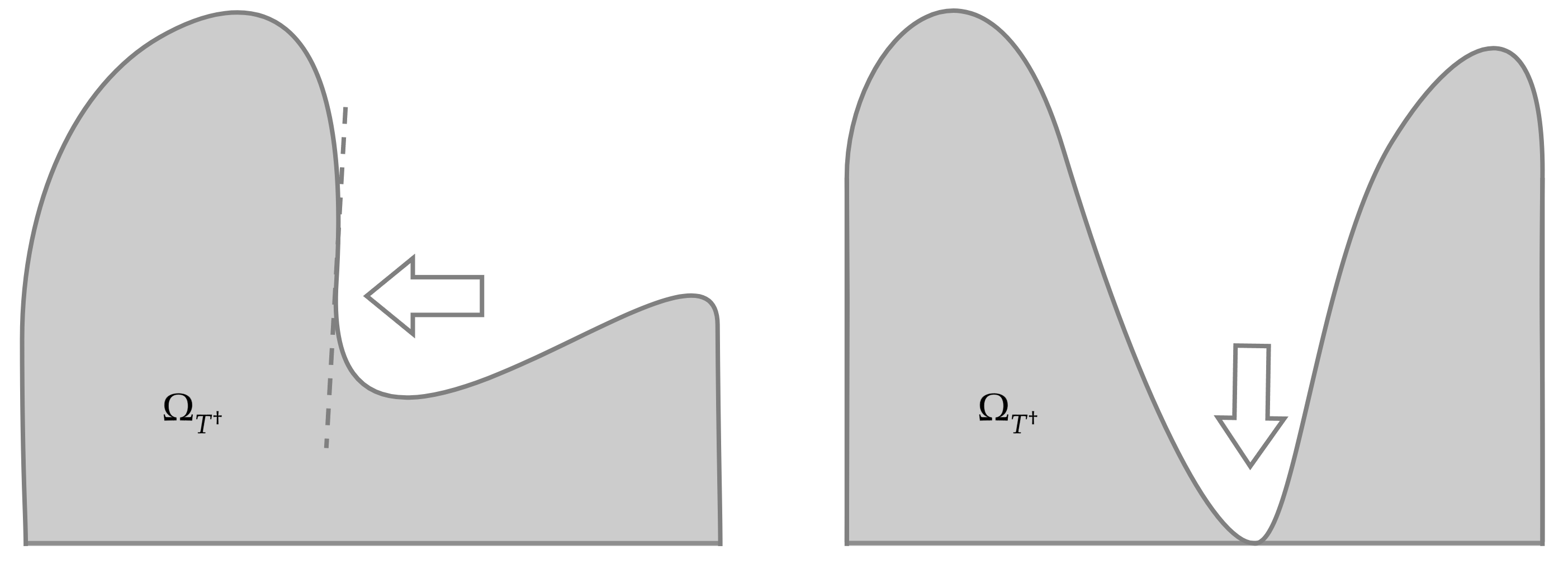}
\caption{Two scenarios under the graph assumption on the torus $\mathbb{T}^2$: on the left, boundary turning occurs, and on the right, the upper free boundary contacts the bottom.
}
\label{fig:turnandcontact}
\end{figure}

To characterize self-intersection of the free boundary, the first and third authors introduced a dynamically updated reference-surface method in \cite{Hao2025}, which preserves non-degenerate coordinate maps as the interface approaches self-contact. That framework gives a blow-up criterion for high-regularity solutions (of $H^6$ class) and captures both boundary and interior singularities. However, the interior singularity condition (cf. (4) in \cite[Theorem 1.2]{Hao2025}) still relies on
\begin{equation*}
\sup_{0\le t< \Tstar}\norm{\nabla v(\cdot,t)}_{H^3(\Omega_t)},
\end{equation*}
so interior blow-up detection remains supremum-based in that formulation.

In this paper, we establish a blow-up criterion for strong solutions in a lower-regularity framework, with velocity in $H^3$. This regularity level is natural in view of the Shatah--Zeng local well-posedness theory for problem \eqref{eqEuler}: the a priori estimates were derived in \cite{Shatah2008,Shatah2008a}, and the existence of solutions with $v\in H^3$ and $\parOmega_t\in H^4$ was established in \cite[Theorem B]{Shatah2011}. Our criterion is formulated at this regularity level in general bounded domains with closed free boundaries. It requires no periodicity, symmetry, simple connectedness, or graph representation, and it includes self-intersection as a genuine boundary blow-up mechanism. Under the additional assumption of simple connectedness, we show that interior blow-up alternative is governed solely by the
time-integrated $L^\infty$-norm of the vorticity, recovering the classical Beale--Kato--Majda (BKM) criterion \cite{Beale1984,Ferrari1993} in the fixed-boundary setting.

\subsection{Main results} 
We now state the main results of the paper.

\begin{theorem}\label{thmain1}
Let $v_0 \in H^3(\Omega_0; \mathbb{R}^3)$ be the initial divergence-free velocity field, where $\Omega_0 \subset \mathbb{R}^3$ is a bounded initial domain (not necessarily simply connected) with a non-self-intersecting closed boundary of class $H^4$. Let $(v, \Omega_t)$ be the solution of the free-boundary problem \eqref{eqEuler} with initial data $(v_0, \Omega_0)$ and maximal existence interval $\left[0, \Tstar\right)$ satisfying
\begin{equation*}
v\in C([0,\Tstar); H^3(\Omega_t))\ \text{and}\ \parOmega_t\in C([0,\Tstar); H^4). 
\end{equation*} 

If $\Tstar< \infty$, then at least one of the following scenarios must occur:
\begin{enumerate}[label={\rm (\arabic*)},topsep=5pt,itemsep=0pt] 
\item Geometric singularity: the first self-intersection of $\parOmega_t$ at $t=\Tstar$.
\item Boundary regularity loss: for any sufficiently small constant $\varepsilon > 0$ independent of $\Tstar$,
\begin{equation*}
\limsup_{t \nearrow \Tstar}\Brace{\norm{\mc_{\parOmega_t}(t)}_{H^{3/2}(\parOmega_t)} +\norm{\parOmega_t}_{H^{2+\varepsilon}}}= \infty.
\end{equation*}
\item Kinematic breakdown: 
\begin{equation*}
\limsup_{t \nearrow \Tstar}\norm{v_n(t)}_{H^{5/2}(\parOmega_t)} = \infty.
\end{equation*} 
\item Interior accumulation of velocity gradients:
\begin{equation}\label{eqBKM1}
\int_{0}^{\Tstar} \norm{\nabla v}_{L^\infty(\Omega_t)}dt=\infty.
\end{equation}
\end{enumerate}
\end{theorem}   
\begin{remark} 
The boundary mechanisms (1)--(3) in Theorem \ref{thmain1} are not meant to be mutually exclusive. For example, at a self-intersection time in Case (1), the free boundary may remain regular along each approaching branch, but it is no longer a single embedded surface at the contact point. Consequently, the unit normal, the mean curvature, and the scalar normal velocity have only branchwise traces there, rather than classical single-valued quantities. In this sense, self-intersection is a degenerate geometric situation in which the classical pointwise quantities appearing in Cases (2) and (3) are lost at the contact point. Moreover, geometric regularity, encoded by curvature, depends on tangential boundary smoothness, whereas the normal velocity governs the kinematics of the boundary motion. Neither Case (2) nor Case (3) forces self-intersection, since curvature concentration, cusps, corners, or high-frequency boundary oscillations may occur while the boundary remains embedded. Conversely, as illustrated below, self-intersection may occur while the curvature remains uniformly controlled before the self-intersection time.
\end{remark} 
\begin{figure}[htb]
\centering
\includegraphics[width=0.8\linewidth]{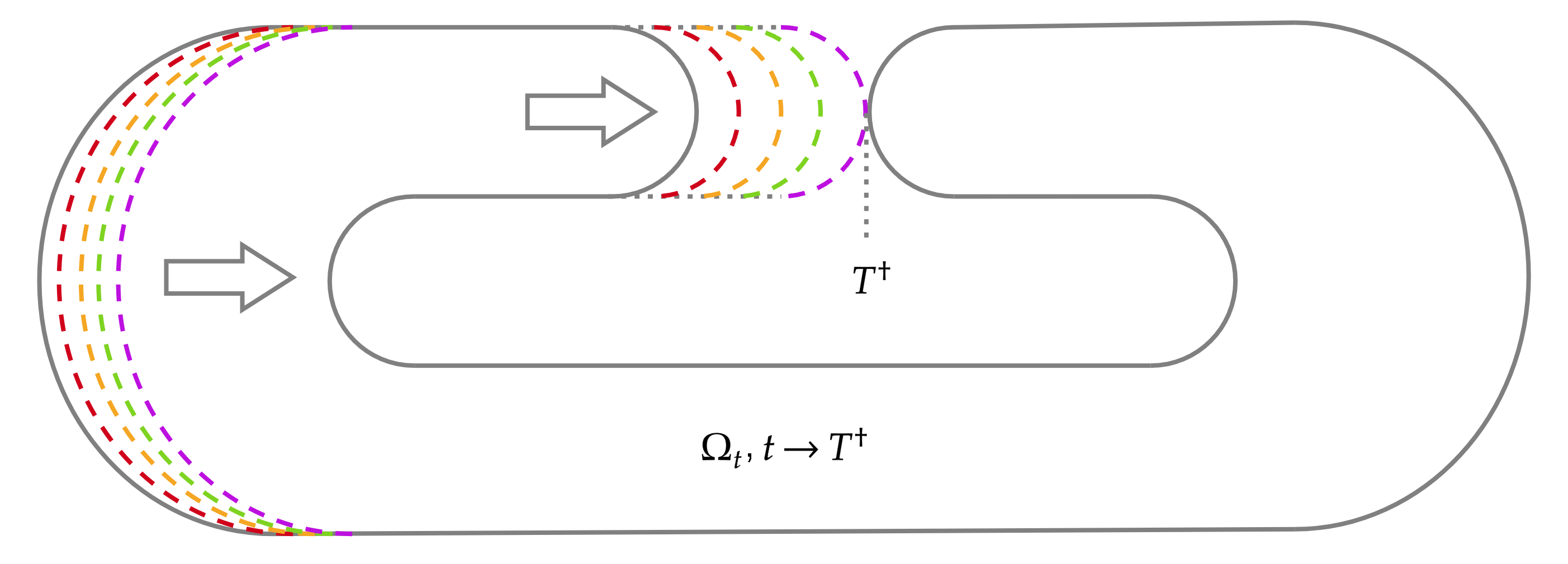}
\caption{Self-intersection with bounded curvature variation.}
\label{fig:ERzeroSFFbounded1}
\end{figure} 
We now illustrate that self-intersection singularities need not be tied to curvature blow-up. In gravity-driven collisions, such as ocean wave impact, the upper fluid layer may descend toward the lower part of the domain without significant curvature amplification before contact. This gives rise to \textit{splash/splat singularities}, where self-intersection occurs with bounded curvature variation, as constructed for 3D Euler equations under Rayleigh--Taylor stability in \cite{Coutand2014} (cf. Fig.~\ref{fig:ERzeroSFFbounded1}). In such examples, curvature norms may grow but remain finite up to self-intersection, as verified for two-dimensional water waves with and without surface tension \cite{Castro2012b,Castro2013}. 

Conversely, in \textit{squeeze singularities} (Fig.~\ref{fig:IRzeroSFFunbounded} and Fig.~\ref{fig:IRzeroSFFbounded}), boundary self-intersection may coincide with curvature blow-up when fluid is extruded through narrowing channels (Fig.~\ref{fig:IRzeroSFFunbounded}). Although no construction is currently available for solutions in fluid domains of the type shown in Fig.~\ref{fig:IRzeroSFFunbounded}, one may expect such a construction to be possible using the domain-perturbation methods developed in \cite{Coutand2014,Coutand2019}. We also mention the recent construction of \textit{splash-squeeze} self-intersection singularities for the two-dimensional plasma-vacuum problem in \cite{Cordoba2025}, where analytic regularity is studied in addition to Sobolev regularity. Squeeze-induced intersection can also preserve curvature regularity when contact occurs along flat interfaces (Fig.~\ref{fig:IRzeroSFFbounded}). For the free-boundary problem \eqref{eqEuler}, the existence of such self-intersecting singular solutions follows from \cite{LuoYang2026} in the case where the magnetic field vanishes.
\begin{figure}[htb]
\centering
\begin{minipage}{0.4\linewidth}
\centering
\includegraphics[width=\linewidth]{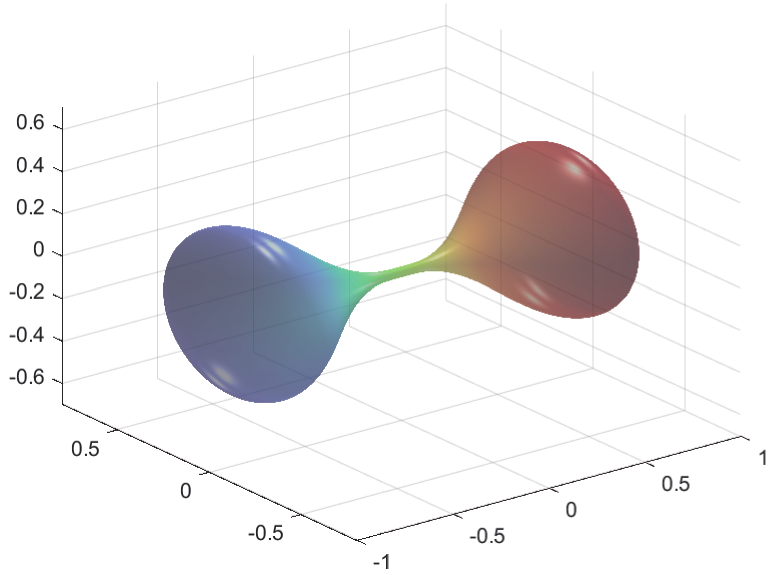} 
\end{minipage}
\begin{minipage}{0.4\linewidth}
\centering
\includegraphics[width=\linewidth]{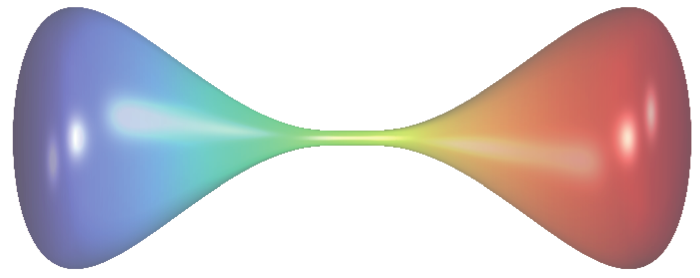} 
\end{minipage} 
\caption{Curvature blow-up accompanying the self-intersection of the free boundary.}
\label{fig:IRzeroSFFunbounded}
\end{figure}
\begin{figure}[htb]
\centering
\begin{minipage}{0.4\linewidth}
\centering
\includegraphics[width=\linewidth]{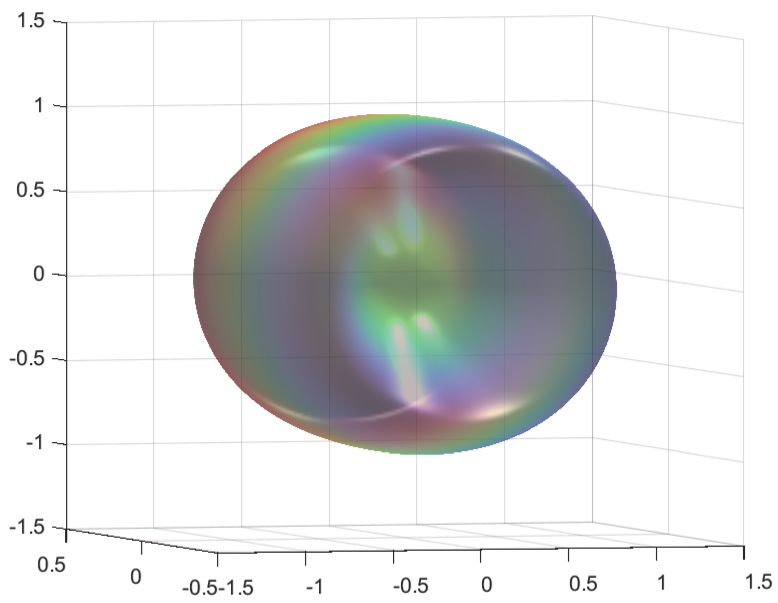} 
\end{minipage}
\begin{minipage}{0.4\linewidth}
\centering
\includegraphics[width=\linewidth]{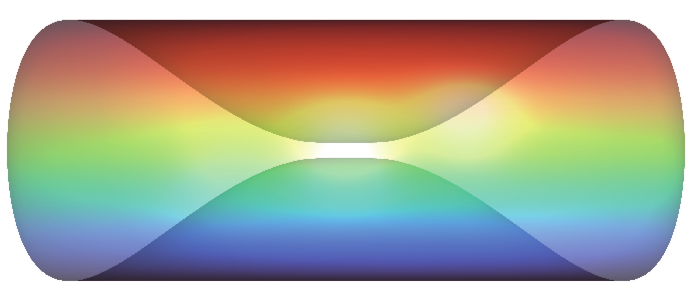} 
\end{minipage} 
\caption{The free boundary approaches self-intersection without curvature blow-up.}
\label{fig:IRzeroSFFbounded}
\end{figure}  
\begin{remark}
We also note the recent work \cite{Ifrim2025} on the free-boundary Euler equations without surface tension. In that setting, the dynamic boundary condition is $p=0$, and the Taylor coefficient plays a central role in both the well-posedness theory and the continuation criterion. In contrast, the present problem has the surface-tension boundary condition \eqref{eqEuler3}, so the pressure trace is prescribed by the mean curvature of the moving boundary. This changes the boundary estimates and leads to a blow-up criterion involving curvature regularity, normal velocity regularity, and possible loss of injectivity of the closed free boundary. Thus, although both works treat the moving boundary intrinsically and do not rely on a global graph representation, the analytic mechanisms and the quantities appearing in the breakdown criteria are different.
\end{remark}
\begin{remark} 
Cases (2) and (3) describe possible boundary regularity loss (Fig.~\ref{fig:case2and3}). The boundary regularity requirements in Case (2) arise from two constraints:
\begin{enumerate}[label={\textup{(\roman*)}},topsep=0pt,itemsep=0pt]
\item $H^2$-regularity of $\parOmega_t$ is the minimal regularity at which the mean curvature is classically defined in \eqref{eqEuler3};
\item recovering boundary regularity from the mean curvature via Lemma \ref{leboundary regularity} requires the slightly stronger $H^{2+\varepsilon}$-regularity, with $\varepsilon > 0$.
\end{enumerate}
This $\varepsilon$-gap reflects the difference between elliptic regularity recovery and the minimal geometric regularity needed to define the curvature. Our purpose here is not to optimize the boundary regularity thresholds in Cases (2) and (3). In this direction, \cite{LuoYang2026} obtains weaker regularity conditions on the curvature and the normal velocity, but at the cost of a stronger time-uniform requirement in the interior blow-up criterion in Case (4).
\begin{figure}[htb]
\centering
\includegraphics[width=0.8\linewidth]{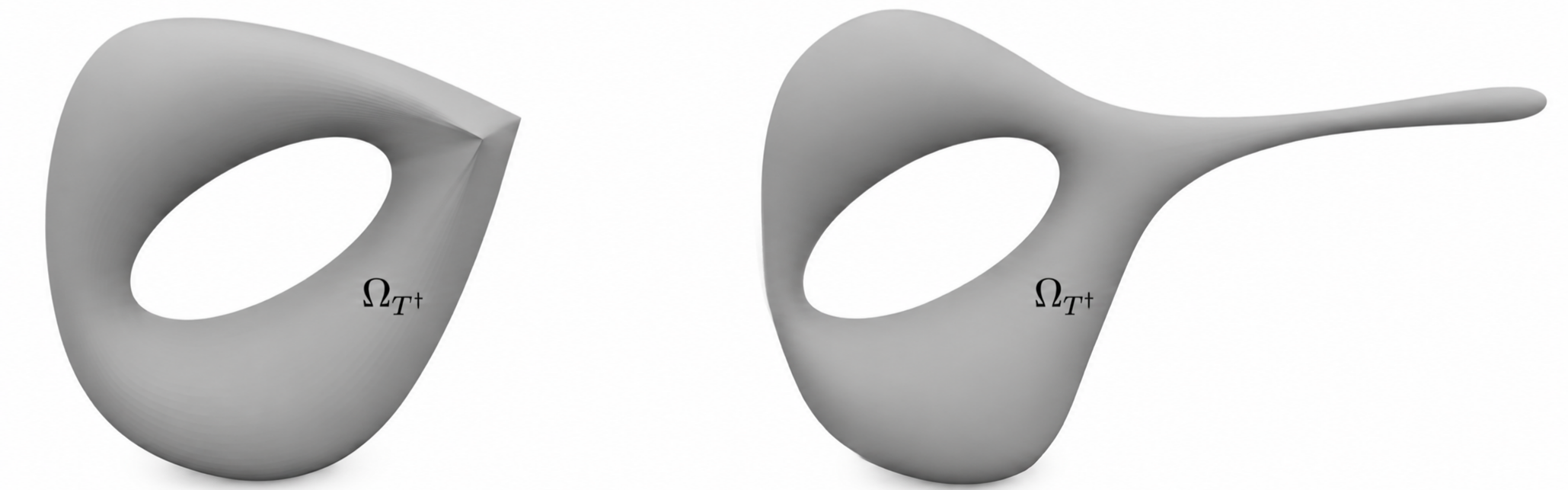}
\caption{Possible scenarios of regularity loss for the free boundary and its normal velocity.}
\label{fig:case2and3}
\end{figure}
\end{remark}

\begin{remark}
The present $H^3\times H^4$ framework is tied to the energy level of the Shatah--Zeng local theory \cite{Shatah2008,Shatah2008a,Shatah2011}, where the nonlinear estimates and the geometric control of the moving boundary can be closed. At this level, the main difficulty is not merely to control the velocity in $H^3$, but to recover this control on a moving closed domain without graph coordinates or a simple-connectedness assumption, while still allowing self-intersection as a possible blow-up scenario. In the continuation argument, once Case (1) is excluded by contradiction, the uniform lower bounds on the exterior and interior ball radii preserve the nondegeneracy of the geometry up to the maximal time. The estimates are then closed by combining the surface-tension boundary energy, pressure estimates, div-curl control of the velocity field, and recovery of boundary regularity from the mean curvature. Together these yield the uniform $H^3$ control of the velocity and $H^4$ control of the moving boundary required for continuation. Our aim is therefore to formulate a continuation criterion at this natural regularity level, rather than to optimize the lowest possible regularity threshold.
\end{remark}

\begin{remark}\label{rmk:fixed-boundary}
For fixed-boundary problems, where $\Omega_t \equiv \Omega_0$ and $v_n \equiv 0$ on $\parOmega_0$ for all $t$, Theorem \ref{thmain1} simplifies as follows:
\begin{enumerate}[label={\textup{(\roman*)}},topsep=0pt,itemsep=0pt]
\item Cases (1)--(3) become trivial, since the geometric quantities are time-independent by assumption.

\item Lemma \ref{le5} no longer requires tracking the evolution of the mean curvature, and the boundary energy term
\begin{equation*}
\int_{\parOmega_t}\module{\Bdnabla\Paren{\DT v\cdot \normal}}^2dS
\end{equation*}
in \eqref{eqenergy1} applied in the proof of Theorem \ref{thmain1} can be dropped.
\end{enumerate}
The blow-up criterion therefore reduces to accumulation of the interior velocity gradient:
\begin{equation}\label{eqBKM3}
\int_{0}^{\Tstar}\norm{\nabla v}_{L^\infty(\Omega_0)} dt = \infty.
\end{equation}
\end{remark}

Under the assumption of simple connectedness, scenario (4) admits
the following BKM-type refinement. 
\begin{theorem}\label{thmain2}
Let $\Tstar< \infty$ denote the maximal existence time in Theorem \ref{thmain1}, and assume that $\Omega_t$ is simply connected. Then the interior blow-up criterion \eqref{eqBKM1} in Theorem \ref{thmain1} admits the vorticity-based refinement
\begin{equation}\label{eqBKM2}
\int_{0}^{\Tstar} \norm{\vorticity v}_{L^\infty(\Omega_t)}dt=\infty.
\end{equation}
Furthermore, in the irrotational case ($\vorticity v \equiv 0$), the blow-up criterion reduces to the boundary alternatives (1)--(3) of Theorem \ref{thmain1}, since the left-hand side of \eqref{eqBKM2} vanishes.
\end{theorem}   
In the simply connected fixed-boundary case, viewed through \eqref{eqBKM2}, the sole blow-up scenario \eqref{eqBKM3} refines to the vorticity-dominated condition
\begin{equation}\label{eqBKM5}  
\int_{0}^{\Tstar} \norm{\vorticity v}_{L^\infty(\Omega_0)}dt=\infty,
\end{equation}
which coincides exactly with the classical BKM criterion \cite{Beale1984,Ferrari1993}.  Thus, together with
Theorem \ref{thmain1}, this refinement recovers the classical
interior blow-up characterization when the boundary is fixed, while
preserving the geometric and boundary singularity mechanisms specific to the free-boundary problem.

Compared with the graph-based criterion in \cite{Luo2024}, which also reduces to the classical BKM condition \eqref{eqBKM5} in the fixed-boundary case, our results provide a more intrinsic and simpler characterization of free-boundary breakdown. Indeed, in simply connected graph domains, the criterion
\eqref{eqluo2024thm1.7b}--\eqref{eqluo2024thm1.7c} in
\cite{Luo2024} involves several graph-dependent quantities and
constraints. In contrast, Theorems \ref{thmain1} and
\ref{thmain2} are formulated directly in terms of the geometry and
motion of general closed free boundaries, without relying on graph
representations, periodicity, or symmetry. Moreover, the regularity
requirements in Cases (2) and (3) are strictly weaker than the
$C^3$-regularity requirements on $\psi(t)$ and
$\partial_t\psi(t)$ in \eqref{eqluo2024thm1.7a1}, respectively. In addition, our analysis does not rely on quantities such as 
\begin{equation*}
\norm{\partial^2_{tt}\psi(t)}_{H^{3/2}},\quad \int_0^t\norm{(v_1,v_2)(s)}_{\dot{W}^{1,\infty}}ds,\quad \norm{(v_1,v_2)(t)}_{L^{\infty}},
\end{equation*} 
nor on the graph-specific condition \eqref{eqluo2024thm1.7c}.

\subsection{Comparison with graph-based, boundary-flattening, and fixed-boundary approaches}

Different methods have different strengths. For the present problem, our analytical framework addresses three limitations of conventional approaches to free-boundary problems:
\begin{enumerate}[label={\textup{(\roman*)}},topsep=0pt,itemsep=0pt] 
\item \textit{Intrinsic geometric formulation}: Graph-based assumptions and local boundary flattening (via coordinate mappings $F$) impose coordinate-dependent restrictions on the geometry. In the presence of surface tension, such reductions may obscure the direct role of the mean curvature in boundary condition \eqref{eqEuler3}, or introduce additional nonlinear geometric terms after transforming the equations. The Eulerian formulation used here keeps the geometry of the moving curved boundary in the estimates and is therefore well suited to tracking curvature-driven boundary mechanisms.
\item \textit{Compatibility with moving-boundary dynamics}: Fixed-boundary formulations are less directly adapted to mechanisms that depend on the motion and geometry of the free surface. By contrast, the Eulerian framework used here keeps the boundary dynamics explicit. It couples curvature evolution to the fluid motion (Lemma \ref{leboundary regularity}), reflects the $3/2$-scaling laws through material derivatives in the energy functional (Remark \ref{rem3/2}), and incorporates the boundary-driven mechanisms in Cases (1)--(3) of Theorem \ref{thmain1}.
\item \textit{Boundary singularity mechanisms}: Boundary-flattening and fixed-boundary reductions are not designed to describe all geometric boundary mechanisms considered here, in particular self-intersection of the free surface and loss of regularity of the normal velocity. The Eulerian formulation allows these mechanisms to be stated directly in the blow-up criterion.
\end{enumerate}

\subsection{Paper structure}
The paper is organized as follows. Section \ref{se2} establishes the auxiliary analytical tools used throughout the argument. Section \ref{se3} proves Theorem \ref{thmain1}, using uniform exterior and interior ball radii to exclude free-boundary self-intersection in the continuation argument. Section \ref{se5} treats simply connected domains and proves the BKM-type blow-up characterization in Theorem \ref{thmain2}.

\section{Tailored auxiliary results}\label{se2}

In this section, we present some fundamental results.  We will adopt the Einstein summation convention and utilize the notation $\tencon$ to denote the contraction of certain indices of tensors with constant coefficients (see, e.g., \cite{Hamilton1982, Mantegazza2002}). 

Using the unit outer normal vector $\normal$, we define the tangential derivative of a scalar function $f$ by 
\begin{equation}\label{eqtangentialderivative}
\Bdnabla f \coloneqq \nabla f - \Paren{\nabla f \cdot \normal} \normal.
\end{equation} 
Similarly, for a vector field $F=\Paren{F^1,F^2,F^3}$, the tangential gradient and the tangential divergence are defined by 
\begin{equation*}
\Bdnabla F= \nabla F - \nabla F(\normal\otimes\normal),\quad \Bddiv F \coloneqq \operatorname{Tr}\Paren{\Bdnabla F}.
\end{equation*}  
The $\Paren{i,j}$-th component of the tangential gradient (with $i, j\in \{1, 2, 3\}$)  is given by
\begin{equation*}
\Bdnabla_j F^i=\Paren{\Bdnabla F}_{ij} = \partial_j F^i - \partial_l F^i \normal^l \normal_j. 
\end{equation*}  
With these notations, the mean curvature of $\parOmega_t$ is expressed as
\begin{equation*}
\mc_{\parOmega_t} = \Bddiv \normal.
\end{equation*}  

Furthermore, fix a point  $x\in\parOmega_t$, and let $X,Y\in T_x\Paren{\parOmega_t}$, where  $T_x\Paren{\parOmega_t}$ denotes the tangent space to $\parOmega_t$ at $x$. The second fundamental form is defined by\footnote{The second fundamental form, when defined via the shape operator, typically depends on the choice of orientation \cite[Section 5.2]{Tu2017}, and a minus sign may be included accordingly. We have chosen not to include the minus sign.}
\begin{equation*}
\textbf{\uppercase\expandafter{\romannumeral2}}(X,Y)=\nabla_{X}\normal\cdot Y= X^j\partial_j\normal\cdot Y.
\end{equation*}
Since  $X\cdot\normal=0$, it follows that
\begin{equation*}
\textbf{\uppercase\expandafter{\romannumeral2}}(X,Y)=\Paren{\partial_j\normal^i-\partial_l\normal^i\normal^l\normal_j} X^jY_i=\sff X\cdot Y=X^\top \sff Y,
\end{equation*}
where $i, j, l \in \{1, 2, 3\}$ and $\sff$ is defined by the tangential gradient of $\normal$:
\begin{equation*}
\sff=\Bdnabla\normal,\quad \sff_{ij}=\partial_j\normal^i-\partial_l\normal^i\normal^l\normal_j.
\end{equation*}  
Following the convention in \cite{Fusco2020}, we refer to the matrix $\sff$ as the second fundamental form. 

\begin{lemma}\label{le1} 
Let $f$ be a smooth function, and  $i, j \in \{1, 2, 3\}$. Then, the following holds
\begin{align*}
&[\DT,\nabla]f=-(\nabla v)^\top\nabla f,\quad  [\DT,\partial_i]f=- \partial_i v^k \partial_k f,\\ 
&[\DT, \nabla^2] f =\nabla v\tencon \nabla^2f+\nabla^2 v\tencon\nabla f,\\
&[\DT,\Bdnabla]f=-(\Bdnabla v)^\top\Bdnabla f,\quad [\DT,\Bdnabla_i]f=-\Bdnabla_i v^k\Bdnabla_k f, \\
&[\DT,\Bdnabla^2]f=\Bdnabla v\tencon \Bdnabla^2 f+\Bdnabla^2 v\tencon \Bdnabla f,\\
&\DT \normal=-(\Bdnabla v)^\top \normal,\quad \DT \normal_i=-\Bdnabla_iv^j \normal_j,\\
&\LB \normal =-\module{\sff}^2\normal+\Bdnabla \mc,\\
&\DT \sff_{ij}=-\Bdnabla_j\Bdnabla_iv^k \normal_k-\Bdnabla_iv^k \sff_{kj}-\Bdnabla_j v^k\sff_{ik},\\
&\DT \mc =-\LB v_n-\module{\sff}^2v_n+\Bdnabla \mc\cdot v,
\end{align*}
where $\Bracket{\cdot,\cdot}$ denotes the Lie bracket, $\LB \coloneqq \Bddiv \Bdnabla$
represents the Beltrami--Laplacian operator, and we abbreviate $\mc = \mc_{\parOmega_t}$.
\end{lemma}	
\begin{proof}
Most of the above formulas can be found in \cite[Section 3.1]{Shatah2008a}. The remaining ones follow from direct calculations. For example, since $\sff_{ij} = \Bdnabla_j \normal_i$, we have 
\begin{align*}
\DT \sff_{ij}&=\Bdnabla_j\DT\normal_i+[\DT,\Bdnabla_j]\normal_i\\
&=\Bdnabla_j\Paren{-\Bdnabla_iv^k \normal_k}-\Bdnabla_j v^k\Bdnabla_k \normal_i\\
&=-\Bdnabla_j\Bdnabla_iv^k \normal_k-\Bdnabla_iv^k \Bdnabla_j\normal_k-\Bdnabla_j v^k\Bdnabla_k \normal_i\\
&=-\Bdnabla_j\Bdnabla_iv^k \normal_k-\Bdnabla_iv^k \sff_{kj}-\Bdnabla_j v^k\sff_{ik}.\qedhere
\end{align*}
\end{proof} 

By divergence-free condition \eqref{eqEuler2}, it follows that 
\begin{equation}\label{eqdivDtv}
\divergence \DT v = \partial_i v^j \partial_j v^i,
\end{equation}
and therefore, from \eqref{eqEuler1}, we obtain
\begin{equation}\label{eqlaplace p}
-\Delta p = \partial_i v^j \partial_j v^i.
\end{equation}

Applying Lemma \ref{le1} to the commutator for the pressure yields the following result. 
\begin{lemma}
For pressure $p$ and velocity $v$ satisfying \eqref{eqEuler1}, the following commutator identities hold:
\begin{equation}\label{eqDtlnablap}
[\DT,\partial_j]p=\partial_j v_i\DT v^i,\quad 
[\DT^{2},\partial_j]p=2\partial_j v_i\DT^2 v^i+\partial_j \DT v_i\DT v^i,
\end{equation} 
for $j = 1,2,3$.
\end{lemma}
\begin{proof}
The first identity follows directly from the commutator $[\DT,\partial_j]$ in Lemma \ref{le1} and the Euler equation \eqref{eqEuler1}. For the second identity, applying the first result and again utilizing the commutator $[\DT,\partial_j]$ from Lemma \ref{le1}, we obtain 
\begin{align*}
[\DT^{2}, \partial_j ]p={}&\DT\Paren{[\DT,\partial_j ]p}+[\DT,\partial_j]\DT p\\
={}&\DT\Paren{\partial_j v_i\DT v^i}-\partial_j v^k\partial_k\DT p\\
={}&\partial_j v_i\DT^2 v^i+\DT\partial_j v_i\DT v^i-\partial_j v^k[\partial_k,\DT] p-\partial_j v^k\DT\partial_k p\\
={}&\partial_j v_i\DT^2 v^i+[\DT,\partial_j] v_i\DT v^i+\partial_j \DT v_i\DT v^i+\partial_j v^k\partial_kv^i\DT v_i+\partial_j v^k\DT^2v_k\\
={}&2\partial_j v_i\DT^2 v^i-\partial_jv^k\partial_kv_i\DT v^i+\partial_j \DT v_i\DT v^i+\partial_j v^k\partial_kv^i\DT v_i\\
={}&2\partial_j v_i\DT^2 v^i+\partial_j \DT v_i\DT v^i.
\end{align*}
This completes the proof. 
\end{proof}  

For a vector field $F$, we define
\begin{equation*}
\vorticity F \coloneqq \nabla F - (\nabla F)^\top.
\end{equation*}
A straightforward calculation also yields the following identity: 
\begin{equation}\label{eqDtvorticity}
[\DT,\vorticity] F=(\nabla v)^\top(\nabla F)^\top-\nabla F\nabla v.
\end{equation} 

To establish the energy estimate in Section \ref{se3}, it is necessary to compute the exact expressions for the following quantities. 

\begin{lemma}\label{le3}  
For the vorticity $\vorticity v$, we have 
\begin{equation*}
\DT \nabla^2 \Paren{\vorticity v}=\nabla v  \tencon \nabla^2 \Paren{\vorticity v}+ \Paren{\vorticity v}\tencon\nabla^{3}v +\nabla^2v  \tencon \nabla \Paren{\vorticity v},
\end{equation*}
where the symbol $\tencon$ denotes the contraction of specific tensor indices as previously stated. 

Moreover, the following identities are valid: 
\begin{align*}
\divergence \DT^2 v&= 3\partial_i v^j \partial_j\DT v^i-2\partial_i v^j\partial_jv^k\partial_kv^i,\\
\divergence\DT^3 v&=4\partial_i v^j \partial_j\DT^2 v^i+3\partial_i\DT v^j \partial_j\DT v^i-12\partial_i v^l\partial_l v^j \partial_j\DT v^i+6\partial_iv^l\partial_l v^j\partial_jv^k\partial_kv^i.
\end{align*}
\end{lemma}
\begin{proof}

Since, by \eqref{eqEuler1},
\begin{equation*}
\vorticity \DT v = 0,
\end{equation*}
it follows that
\begin{equation*}
\DT \Paren{\vorticity v} = \DT \vorticity v - \vorticity \DT v = [\DT, \vorticity] v,
\end{equation*}
and by applying the identity \eqref{eqDtvorticity}, we obtain
\begin{equation}\label{eqDtcurlv}
\DT \Paren{\vorticity v} = (\nabla v)^\top (\nabla v)^\top - \nabla v \nabla v= -(\nabla v)^\top \Paren{\vorticity v} - \Paren{\vorticity v} \nabla v. 
\end{equation} 

Then, the first claim follows immediately from the commutator $[\DT, \nabla^2]$ in Lemma \ref{le1}. Indeed, applying \eqref{eqDtcurlv}, we obtain 
\begin{align*}
\DT\nabla^2\Paren{\vorticity v}
&=\nabla^2\DT\Paren{\vorticity v}+[\DT,\nabla^2]\Paren{\vorticity v}\\
&=\nabla^2\Bracket{-(\nabla v)^{\top}\Paren{\vorticity v}-\Paren{\vorticity v}\nabla v}+\nabla v\tencon \nabla^2 \Paren{\vorticity v}+\nabla^2 v\tencon\nabla \Paren{\vorticity v}\\
&=\nabla v  \tencon \nabla^2 \Paren{\vorticity v}+ \Paren{\vorticity v}\tencon\nabla^{3}v +\nabla^2v  \tencon \nabla \Paren{\vorticity v}.
\end{align*}

Next, using divergence-free condition \eqref{eqEuler2}, the commutator $[\DT, \partial_j]$ from Lemma \ref{le1}, and the definition of the material derivative $\DT = \partial_t + v^k \partial_k$, we obtain 
\begin{align*}
\divergence \DT^2 v={}&\partial_i\Paren{\partial_t\DT v^i+v^j\partial_j\DT v^i}\\
={}&\partial_i\Bracket{\partial_t\Paren{\partial_t v^i+v^k\partial_k v^i}+v^j\partial_j\Paren{\partial_t v^i+v^k\partial_k v^i}}\\
={}&\partial_i\Bracket{\partial^2_t v^i+\partial_tv^k\partial_k v^i+v^k\partial_t\partial_k v^i+v^j\Paren{\partial_t\partial_j v^i+\partial_jv^k\partial_k v^i+v^k\partial^2_{jk} v^i}}\\
={}&\partial_t\partial_iv^k\partial_k v^i+\partial_iv^k\partial_t\partial_k v^i+\partial_iv^j\Paren{\partial_t\partial_j v^i+\partial_jv^k\partial_k v^i+v^k\partial^2_{jk} v^i}\\
&+v^j\Paren{\partial^2_{ij}v^k\partial_k v^i+\partial_iv^k\partial^2_{jk} v^i}\\
={}&3\partial_t\partial_jv^i\partial_i v^j+3\partial_iv^jv^k\partial^2_{jk} v^i+\partial_iv^j\partial_jv^k\partial_k v^i\\
={}&3\partial_i v^j\DT \partial_jv^i+\partial_iv^j\partial_jv^k\partial_k v^i\\
={}&3\partial_i v^j \partial_j\DT v^i+3\partial_i v^j[\DT, \partial_j]v^i+\partial_iv^j\partial_jv^k\partial_k v^i\\
={}&3\partial_i v^j \partial_j\DT v^i-2\partial_i v^j\partial_jv^k\partial_kv^i. 
\end{align*}

Finally, from the above formula and again using the commutator $[\DT, \partial_j]$ in Lemma \ref{le1}, it follows that 
\begin{align*}
\divergence\DT^3 v={}&\DT\divergence\DT^2 v+[\divergence,\DT]\DT^2 v\\
={}&\DT\Bracket{3\partial_i v^j \partial_j\DT v^i-2\partial_i v^j\partial_jv^k\partial_kv^i}+[\partial_j,\DT]\DT^2 v^j\\
={}&3\DT\Paren{\partial_i v^j \partial_j\DT v^i}-6\Paren{\DT\partial_i v^j}\partial_jv^k\partial_kv^i+\partial_jv^i\partial_i\DT^2 v^j\\
={}&3\DT\partial_i v^j \partial_j\DT v^i+3\partial_i v^j\DT \partial_j\DT v^i-6\partial_i\DT v^j\partial_jv^k\partial_kv^i+6\partial_iv^l\partial_l v^j\partial_jv^k\partial_kv^i\\
&+\partial_jv^i\partial_i\DT^2 v^j\\
={}&3\Paren{\partial_i\DT v^j \partial_j\DT v^i-\partial_i v^l\partial_l v^j \partial_j\DT v^i} +3\Paren{\partial_i v^j \partial_j\DT^2 v^i-\partial_i v^j \partial_jv^l\partial_l\DT v^i}\\
&-6\partial_i\DT v^j\partial_jv^k\partial_kv^i+6\partial_iv^l\partial_l v^j\partial_jv^k\partial_kv^i+\partial_jv^i\partial_i\DT^2 v^j\\
={}&4\partial_i v^j \partial_j\DT^2 v^i+3\partial_i\DT v^j \partial_j\DT v^i-12\partial_i v^l\partial_l v^j \partial_j\DT v^i+6\partial_iv^l\partial_l v^j\partial_jv^k\partial_kv^i,
\end{align*}
where we have utilized the following identities: 
\begin{align*}
\DT\partial_i v^j \partial_j\DT v^i&=\partial_i\DT v^j \partial_j\DT v^i-\partial_i v^l\partial_l v^j \partial_j\DT v^i,\\
\partial_i v^j\DT \partial_j\DT v^i&=\partial_i v^j \partial_j\DT^2 v^i-\partial_i v^j \partial_jv^l\partial_l\DT v^i.
\end{align*} 
This concludes the proof.
\end{proof}

To control $I_1(t)$ in Proposition \ref{prenergy estimate}, we extract the divergence part of $-\Delta \DT^2 p$.
\begin{lemma}\label{le4}  
The following identity holds:
\begin{align*}
-\Delta \DT^2 p={}&4\partial_i v^j \partial_j\DT^2 v^i+3\partial_i\DT v^j \partial_j\DT v^i-12\partial_i v^l\partial_l v^j \partial_j\DT v^i+6\partial_iv^l\partial_l v^j\partial_jv^k\partial_kv^i\\
&+\sum_j\partial_j \Paren{2\partial_j v_i\DT^2 v^i+\partial_j \DT v_i\DT v^i}.
\end{align*}
\end{lemma} 
\begin{proof}  
This result is obtained by applying \eqref{eqEuler1}, \eqref{eqDtlnablap}, and Lemma \ref{le3}: 
\begin{align*}
-\Delta \DT^{2}p={}&-\divergence\nabla\DT^2p\\
={}&-\divergence\DT^2\nabla p-\divergence[\nabla,\DT^2]p\\
={}&\divergence \DT^3 v+\sum_j\partial_j [\DT^2,\partial_j]p\\
={}&4\partial_i v^j \partial_j\DT^2 v^i+3\partial_i\DT v^j \partial_j\DT v^i-12\partial_i v^l\partial_l v^j \partial_j\DT v^i+6\partial_iv^l\partial_l v^j\partial_jv^k\partial_kv^i\\
&+\sum_j\partial_j \Paren{2\partial_j v_i\DT^2 v^i+\partial_j \DT v_i\DT v^i}.\qedhere
\end{align*} 
\end{proof} 

To handle the boundary energy, it is necessary to track the evolution of the pressure on the free boundary and isolate the error terms. 
\begin{lemma}\label{le5}
On the free boundary $\parOmega_t$, we have  
\begin{equation*}
\DT^{2} p= -\LB (\DT v\cdot\normal)+\error,
\end{equation*}
with the error terms
\begin{equation*}
\error 
=-\module{\sff}^2 \DT v\cdot\normal+\Bdnabla p\cdot\DT v +\Bdnabla^2v\tencon\Bdnabla v\tencon\normal+\Bdnabla v\tencon\Bdnabla v\tencon \sff. 
\end{equation*} 
\end{lemma}
\begin{proof}
From boundary condition \eqref{eqEuler3}, $\sff_{ij}=\Bdnabla_j\normal_i$ and the identities 
\begin{equation}\label{eqDtmc} 
\DT \mc =-\LB v_n-\module{\sff}^2v_n+\Bdnabla \mc\cdot v,\quad \LB \normal =-\module{\sff}^2\normal+\Bdnabla \mc,\quad 
\end{equation}
in Lemma \ref{le1}, it follows that
\begin{align*}
\DT p&=-\LB v_n-\module{\sff}^2v_n+\Bdnabla \mc\cdot v\\
&=-\LB v_n+\LB \normal\cdot v\\
&=-\LB v\cdot \normal-2 \sff:\Bdnabla v,
\end{align*}
where ``$:$'' indicates double contraction, i.e., $A:B=\sum_{i,j}A_{ij}B_{ij}$ for the matrices $A$ and $B$. 
We then differentiate to obtain 
\begin{equation*}
\DT^2p=-\DT\LB v\cdot \normal -\LB v\cdot \DT\normal-2\DT \sff:\Bdnabla v-2 \sff:\DT\Bdnabla v.
\end{equation*} 
From the formulas for $\DT \normal$, $\DT \sff$, $[\DT, \Bdnabla]$, and $[\DT, \Bdnabla^2]$ in Lemma \ref{le1}, along with $\sff_{ij} = \Bdnabla_j \normal_i$, it follows that 
\begin{align*}
\DT^2p={}&-\LB \DT v^i\normal_i-[\DT,\LB] v^i\normal_i-\LB v^i\DT\normal_i-2\DT \sff_{ij}\Bdnabla_jv^i\\
&-2 \sum_j\sff_{ij} \Bdnabla_j \DT v^i-2\sum_j \sff_{ij}: [\DT,\Bdnabla_j] v^i\\ 
={}&-\sum_j\Bdnabla_j\Bdnabla_j \DT v^i\normal_i-2\sum_j\Bdnabla_j \DT v^i \Bdnabla_j\normal_i-[\DT,\LB] v^i\normal_i+\LB v^i\Bdnabla_i v^j \normal_j\\
&+2\sum_j\Paren{\Bdnabla_j\Bdnabla_iv^k \normal_k+\Bdnabla_iv^k \sff_{kj}+\Bdnabla_j v^k\sff_{ik}}\Bdnabla_jv^i+2\sum_j\sff_{ij}\Bdnabla_jv^k\Bdnabla_kv^i\\ 
={}&-\LB\Paren{\DT v^i\normal_i}+\DT v^i\LB\normal_i+\Bdnabla^2v\tencon\Bdnabla v\tencon\normal+\LB v^i\Bdnabla_i v^j \normal_j\\ 
&+2\sum_j\Bdnabla_j\Bdnabla_iv^k \Bdnabla_jv^i\normal_k+2\sum_j\Bdnabla_j v^k\Bdnabla_jv^i\sff_{ik}+4\sum_j\Bdnabla_jv^k\Bdnabla_kv^i\sff_{ij}\\ 
={}&-\LB\Paren{\DT v\cdot\normal}+\DT v^i\Paren{-\module{\sff}^2\normal_i+\Bdnabla_i\mc}+\Bdnabla^2v\tencon\Bdnabla v\tencon\normal+\Bdnabla v\tencon\Bdnabla v\tencon \sff,
\end{align*}
where \eqref{eqDtmc} has been used in the last step, and the proof is complete. 
\end{proof}  

For $u \in L^2(\parOmega)$, we define $u \in H^{\frac{1}{2}}(\parOmega)$ if 
\begin{equation*}
\norm{u}_{H^{\frac 12}(\parOmega)}\coloneqq \norm{u}_{L^2(\parOmega)}+\inf \Brace{\norm{\nabla w}_{L^2(\Omega)}: w\in H^1(\Omega) \text{ and } w|_{\parOmega}=u}<\infty.
\end{equation*} 
It follows immediately from the definition that, for any $u\in H^1(\Omega)$,
\begin{equation}\label{eqH1/2}
\norm{u}_{H^{\frac 12}(\parOmega)}\le \norm{u}_{L^2(\parOmega)}+\norm{\nabla u}_{L^2(\Omega)}.
\end{equation}

The relationship between the regularity of the mean curvature and that of the second fundamental form is expressed as follows: 
\begin{lemma}\label{lecurvaturebound}
Let $\Omega \subset \mathbb{R}^3$ be a bounded domain with $C^{1,\alpha}$ boundary ($\alpha \in (0,1)$). Then, the second fundamental form $\sff$ and mean curvature $\mc$ satisfy 
\begin{equation*}
\norm{\sff}_{L^p(\parOmega)} \le C \Paren{1+\norm{\mc}_{L^p(\parOmega)}},\quad p \in (1, \infty).
\end{equation*}
If, in addition, $\norm{\sff}_{L^4(\parOmega)} \le C$ for some positive constant $C$, then for $k \in\Brace{\frac{1}{2}, 1, \frac{3}{2}, 2}$, we have
\begin{equation*}
\norm{\sff}_{H^{k}(\parOmega)} \le C \Paren{1+\norm{\mc}_{H^{k}(\parOmega)}}.
\end{equation*}  
\end{lemma}
\begin{proof}
The $L^p$ estimate and the $H^k$ estimate for $k = \frac{1}{2}, 1, 2$ can be found in \cite[Proposition 2.12]{Julin2024}, while the case for $k = \frac{3}{2}$ can be obtained via the Sobolev interpolation $H^{\frac{3}{2}} = [H^1, H^2]_{\frac{1}{2}}$. 
\end{proof}

The regularity of the free boundary is fully controlled by its mean curvature through elliptic regularity:
\begin{lemma}\label{leboundary regularity}
Let $\Omega \subset \mathbb{R}^3$ be a domain with $\parOmega \in H^{s_0}$ for some $s_0 > 2$. If the mean curvature satisfies
\begin{equation*}
\norm{\mc}_{H^{s-2}(\parOmega)}<\infty\ \text{for}\ s> s_0,
\end{equation*} 
then the boundary regularity lifts to $\parOmega\in H^s$.
\end{lemma}
\begin{proof}
See \cite[Proposition A.2]{Shatah2008a}.
\end{proof}

This implies that for a time-dependent free boundary $\parOmega_t$, the uniform bound 
\begin{equation*}
\sup_{t \in I} \norm{\mc}_{H^{s-2}(\parOmega_t)} < \infty,\ \text{on the time interval}\ I,
\end{equation*} 
guarantees uniform control of the boundary regularity: $\parOmega_t \in H^s,\forall t\in I$, provided $\parOmega_t \in H^{s_0}$. 

The following div-curl estimates play a pivotal role in the subsequent section. 
\begin{lemma}
Let $\Omega \subset \mathbb{R}^3$ be a bounded domain with $\parOmega \in H^{2+\varepsilon}$ for some $\varepsilon > 0$ and $\norm{\sff}_{H^{\frac{3}{2}}(\parOmega)} \le C$ for some positive constant $C$. Then, for any smooth vector field $F$ and for $k \in \Brace{1, \frac{3}{2}, 3}$, the following holds:
\begin{equation}\label{eqdivcurlk}
\norm{F}_{H^{k}(\Omega)} \le C \Paren{\norm{F_n}_{H^{k-\frac 12}(\parOmega)}+\norm{F}_{L^2(\Omega)}+\norm{\divergence F}_{H^{k-1}(\Omega)}+\norm{\vorticity F}_{H^{k-1}(\Omega)}},
\end{equation} 
where $F_n = F \cdot \normal$. 
\end{lemma}
\begin{proof} 
The $H^{2+\varepsilon}$-regularity of the free boundary implies $C^{1,\alpha}$-regularity for some sufficiently small $\alpha= \alpha(\varepsilon) > 0$.

\textit{Case $k=1$}: Follows from \cite[Theorem 3.6]{Julin2024} via standard div-curl theory.

\textit{Case $k=3$}: By Lemma \ref{leboundary regularity}, $\norm{\sff}_{H^{\frac{3}{2}}(\parOmega)} \le C$ implies $\parOmega \in H^{7/2}$. Adapting \cite[Theorem 3.1]{Julin2024} (based on \cite[Theorem 1.3]{Cheng_2016}):
\begin{align*}
\norm{F}_{H^3(\Omega)}
\le{}& C\Big(\norm{\Bdnabla F_n}_{H^{\frac{3}{2}}(\parOmega)}+ (1+\norm{\sff}_{H^{\frac{3}{2}}(\parOmega)})\norm{F}_{L^{\infty}(\Omega)}\nonumber\\
&\qquad+\norm{\divergence F}_{H^{2}(\Omega)}+\norm{\vorticity F}_{H^{2}(\Omega)}\Big)\\
\le{}& C\Paren{\norm{F_n}_{H^{\frac{5}{2}}(\parOmega)}+ \norm{F}_{L^{\infty}(\Omega)}+\norm{\divergence F}_{H^{2}(\Omega)}+\norm{\vorticity F}_{H^{2}(\Omega)}}.
\end{align*}
By interpolation, we obtain the following estimate:
\begin{equation*}
\norm{F}_{L^\infty(\Omega)} \le \varepsilon \norm{F}_{H^3(\Omega)} + C_\varepsilon \norm{F}_{L^2(\Omega)},
\end{equation*}
for sufficiently small $\varepsilon > 0$. Therefore, we have 
\begin{equation*}
\norm{F}_{H^3(\Omega)}\le C\Paren{\norm{F_n}_{H^{\frac{5}{2}}(\parOmega)}+\norm{F}_{L^2(\Omega)}+\norm{\divergence F}_{H^{2}(\Omega)}+\norm{\vorticity F}_{H^{2}(\Omega)}}.
\end{equation*} 

\textit{Case $k=\frac{3}{2}$}: Follows from interpolation between $H^1$ and $H^3$ cases.
\end{proof}

Finally, we list the specific elliptic estimates that will be utilized in subsequent sections. 
\begin{lemma}\label{leellipticestimate} 
Let $\Omega \subset \mathbb{R}^3$ be a bounded domain with $\parOmega \in C^{1,\alpha}$ for some $\alpha \in (0,1)$, and $\norm{\sff}_{H^{\frac{1}{2}}(\parOmega)} \le C$ for some positive constant $C$. Let $u$ be the solution to the following Dirichlet problem:
\begin{equation*}
\begin{cases}
\Delta u=f, & \text{ in } \Omega,\\ 
u=0, & \text{ on } \parOmega.
\end{cases}
\end{equation*}
Then, it holds
\begin{equation*}
\norm{\partial_\normal u}_{H^1(\parOmega)}+\norm{\nabla u}_{H^{\frac{3}{2}}(\Omega)} \le C\norm{f}_{H^{\frac{1}{2}}(\Omega)},
\end{equation*}
where $\partial_\normal$ denotes the outer normal derivative.

Moreover, if the function $f$ can be expressed as $f = \divergence F$ for some vector field $F$, then we have
\begin{equation*}
\norm{u}_{H^1(\Omega)}\le C\norm{F}_{L^2(\Omega)}.
\end{equation*}
\end{lemma}
\begin{proof}
See \cite[Proposition 3.8]{Julin2024} for the first elliptic estimate. We now demonstrate the second one. Using integration by parts, we obtain
\begin{equation*}
\int_{\Omega}\module{\nabla u}^2dx=-\int_{\Omega}\Delta uudx=-\int_{\Omega}\divergence Fudx=\int_{\Omega}F^i\partial_iudx.
\end{equation*}
It follows that 
\begin{equation*}
\norm{\nabla u}_{L^2(\Omega)}^2\le \norm{\nabla u}_{L^2(\Omega)}\norm{F}_{L^2(\Omega)},
\end{equation*}
and hence
\begin{equation*}
\norm{\nabla u}_{L^2(\Omega)}\le \norm{F}_{L^2(\Omega)}.
\end{equation*}
Therefore, the $H^1$ estimate follows by Poincar\'{e}'s inequality. 
\end{proof}

\section{Proof of Theorem \ref{thmain1}}\label{se3}

We establish Theorem \ref{thmain1} by contradiction. 

Suppose that the maximal time $\Tstar <\infty$. Then, either the velocity field $v(\cdot, \Tstar) \notin H^3(\Omega_{\Tstar})$, or the free boundary $\parOmega_{\Tstar} \notin H^4$. 

Assume that none of the four scenarios in Theorem \ref{thmain1} holds. Then, there exists a positive constant $\CC$ such that
\begin{align}
&\inf_{0\le t<\Tstar}\radi(\Omega_t)>\CC^{-1},\quad\sup_{0\le t<\Tstar}
\norm{\parOmega_t}_{H^{2+\varepsilon}}\le \CC,\label{eqass1}\\
&\sup\limits_{0\le t< \Tstar}\Paren{\norm{\mc}_{H^{\frac 32}(\parOmega_t)}+\norm{v_n}_{H^{\frac 52}(\parOmega_t)}}\le \CC,\label{eqass2}\\
&\int_{0}^{\Tstar} \norm{\nabla v}_{L^{\infty}(\Omega_t)}dt\le \CC,\label{eqass3}\\
&\norm{\vorticity v_0}_{L^2(\Omega_0)}^2+\frac{1}{2}\norm{v_0}_{L^2(\Omega_0)}^2+\mathcal{H}^{2}(\parOmega_0)\le \CC. \label{eqass4}
\end{align} 
Here, $\mathcal{H}^{2}(\parOmega_0):=\int_{\parOmega_0}dS$ is the surface area, i.e., the two-dimensional Hausdorff measure of the surface, and  $\radi(\Omega_t)$ denotes the uniform exterior and interior ball radius of $\Omega_t$, defined as:
\begin{equation*}
\radi(\Omega_t):=\sup  \Brace{r > 0 : \forall x \in \parOmega_t, \ \exists B_r(y) \subset \Omega_t, \ B_r(z) \subset \Omega_t^c\ \text{with}\ x \in \partial B_r(y) \cap \partial B_r(z)}.
\end{equation*} 
Note that $\radi(\Omega_t) > 0$ excludes singularities like cusps, corners, or boundary self-intersection (Case (1)), see Fig.~\ref{fig:case1} for an illustration of the case where the radius $\radi = 0$, in which the boundary self-intersects in at least one point.

We recall the following Reynolds transport theorem (see, e.g.,  \cite{Shatah2008a}).

\begin{lemma}\label{leReynolds}
For a smooth function $f$ defined on the moving domain $\Omega_t$, the following holds:
\begin{align*}
\frac{d}{dt}\int_{\Omega_t}fdx&=\int_{\Omega_t}\DT fdx,\\
\frac{d}{dt}\int_{\parOmega_t}fdS&=\int_{\parOmega_t}(\DT f+f\Bddiv v) dS.
\end{align*}
\end{lemma}

Applying this to the energy components of \eqref{eqEuler} yields
\begin{align*}
\frac{d}{dt}\Paren{\frac{1}{2}\norm{v}_{L^2(\Omega_t)}^2} &= \int_{\Omega_t} v \cdot (-\nabla p)  dx = -\int_{\parOmega_t} p v_n  dS, \\
\frac{d}{dt} \mathcal{H}^2(\parOmega_t) &=\frac{d}{dt}\int_{\parOmega_t}dS=\int_{\parOmega_t}\Bddiv v dS= \int_{\parOmega_t} \mc v_n  dS,
\end{align*}
where the last equality follows from differential geometry for a closed surface, i.e., $$\int_{\parOmega_t}\Bddiv v_{\text{tan}} dS=0,$$ for the tangential velocity $v_{\text{tan}}=v-v_n\normal$. 
The pressure-curvature coupling $p = \mc$ then yields exact energy conservation:
\begin{equation*}
\frac{1}{2}\norm{v}_{L^2(\Omega_t)}^2+\mathcal{H}^{2}(\parOmega_t)\equiv\frac{1}{2}\norm{v}_{L^2(\Omega_0)}^2+\mathcal{H}^{2}(\parOmega_0),\quad 0\le t<\Tstar.
\end{equation*} 
Thus from \eqref{eqass4}, we obtain the uniform bound:
\begin{equation}\label{eqvL2}
\sup\limits_{0\le t< \Tstar}\Paren{\frac{1}{2}\norm{v}_{L^2(\Omega_t)}^2+\mathcal{H}^{2}(\parOmega_t)}\le \CC.
\end{equation} 

\begin{remark}
The constant in the following will primarily depend on $\CC$ and $\CC^{-1}$ (the lower bound of the uniform ball radius), and we will denote it simply as $C(\CC)$.
\end{remark}

Next, we establish uniform vorticity control:
\begin{equation*}
\sup\limits_{0\le t< \Tstar}\norm{\vorticity v}_{L^2(\Omega_t)}\le C(\CC).
\end{equation*} 
In fact, by applying Lemma \ref{leReynolds} to the vorticity energy and using \eqref{eqDtcurlv}, we obtain
\begin{align*}
\frac{d}{dt}\Paren{\frac 12 \norm{\vorticity v}_{L^2(\Omega_t)}^2}={}&\int_{\Omega_t}\DT\Paren{\vorticity v}:\Paren{\vorticity v}dx\\
={}&-\int_{\Omega_t}\Bracket{(\nabla v)^{\top}\Paren{\vorticity v}+\Paren{\vorticity v}\nabla v}:\Paren{\vorticity v}dx\\
\le{}&2\norm{\nabla v}_{L^\infty(\Omega_t)}\norm{\vorticity v}_{L^2(\Omega_t)}^2,\quad 0< t<\Tstar.
\end{align*}
By  \eqref{eqass3} and \eqref{eqass4}, for any $0 \le t < \Tstar$, Gr\"onwall's inequality then yields:
\begin{align}
\norm{\vorticity v}_{L^2(\Omega_t)}^2&\le \norm{\vorticity v_0}_{L^2(\Omega_0)}^2 \exp\Paren{4\int_0^t\norm{\nabla v(s)}_{L^\infty(\Omega_s)}ds}\nonumber\\
&\le  \norm{\vorticity v_0}_{L^2(\Omega_0)}^2 \exp\Paren{4\int_0^{\Tstar}\norm{\nabla v(t)}_{L^\infty(\Omega_t)}dt}\nonumber\\
&\le C(\CC).\label{eqcurlvL2}
\end{align} 

We note that, by  \eqref{eqass1}, the free boundary $\parOmega_t$ belongs to the $C^{1,\alpha}$-class for some $\alpha \in (0,1)$ throughout the time interval $\left[0, \Tstar\right)$. Combining with the curvature bound in  \eqref{eqass2} and Sobolev's embedding, we apply Lemma \ref{lecurvaturebound} to conclude that the second fundamental form $\sff$ satisfies
\begin{equation*}
\norm{\sff}_{L^4(\parOmega_t)}\le C\Paren{1+\norm{\mc}_{L^4(\parOmega_t)}}\le C(\CC),\quad 0\le t< \Tstar.
\end{equation*} 
Then, from Lemma \ref{lecurvaturebound} again, it follows that
\begin{equation}\label{eqsffbound}
\norm{\sff}_{H^{\frac 32}(\parOmega_t)}\le C(\CC)\Paren{1+\norm{\mc}_{H^{\frac 32}(\parOmega_t)}}\le C(\CC),\quad 0\le t< \Tstar.
\end{equation}
Since, by \eqref{eqass1}, the free boundary $\parOmega_t$ belongs to the $H^{2+\varepsilon}$-class throughout the time interval $\left[0,\Tstar\right)$, and by applying Lemma \ref{leboundary regularity}, we deduce that the free boundary $\parOmega_t \in H^{\frac{7}{2}}$ uniformly on $\left[0,\Tstar\right)$, as the constant $C(\CC)$ in \eqref{eqsffbound} is independent of time. 
Additionally, the unit outer  normal vector $\normal \in H^{\frac{5}{2}}$ uniformly on $\left[0,\Tstar\right)$, since $\sff = \Bdnabla \normal$. Using boundary condition \eqref{eqEuler3}, we conclude that
\begin{equation}\label{eqCCC}
\begin{aligned}
&\inf_{0\le t<\Tstar}\radi(\Omega_t)>\CC^{-1},\quad \sup_{0\le t<\Tstar}
\norm{\parOmega_t}_{H^{\frac72}}
\le C(\CC),\\
&\sup\limits_{0\le t< \Tstar}\Paren{\norm{\normal}_{H^{\frac 52}(\parOmega_t)}+\norm{\sff}_{H^{\frac 32}(\parOmega_t)}+\norm{v_n}_{H^{\frac 52}(\parOmega_t)}}\le C(\CC),\\ 
&\sup\limits_{0\le t< \Tstar}\Paren{\norm{\vorticity v}_{L^2(\Omega_t)}+\norm{v}_{L^2(\Omega_t)}}\le C(\CC).
\end{aligned} 
\end{equation}

\begin{remark}
\eqref{eqass1}--\eqref{eqass4} are equivalent to \eqref{eqass3} and \eqref{eqCCC}. However, in the absence of  \eqref{eqass3}, we can derive all the other assumptions in \eqref{eqCCC}, except the $L^2$-bound on the vorticity given by \eqref{eqcurlvL2}, i.e.,
\begin{equation*}
\sup\limits_{0\le t< \Tstar}\norm{\vorticity v}_{L^2(\Omega_t)}\le C(\CC) 
\end{equation*}
may not be valid since the exponential growth factor in Gr\"onwall's estimate depends critically on
\begin{equation*}
\int_0^{\Tstar}\norm{\nabla v(t)}_{L^\infty(\Omega_t)}dt.
\end{equation*}
\end{remark}

We extend the unit outer normal $\normal$ to $\Omega_t$ using harmonic extension. By elliptic regularity theory for Dirichlet problems and the uniform domain characteristics in \eqref{eqCCC}, the extended $\normal$ (retaining the same notation) satisfies:
\begin{equation}\label{eqextendnormal}
\sup_{0\le t< \Tstar}\norm{\normal}_{H^{3}(\Omega_t)}\le C(\CC),\ \text{and}\ \sup_{0\le t< \Tstar} \norm{\normal}_{H^{5/2}(\parOmega_t)} \le C(\CC).
\end{equation}

Under \eqref{eqCCC}, we will derive energy estimates \eqref{eqenergy estimate} and \eqref{eqclose energy} 
on the time interval $(0,\Tstar)$.

We now list some inequalities that will be frequently used. The following Gagliardo--Nirenberg inequality can be found in \cite{Mironescu_2018}.
\begin{lemma}\label{leGN ineq}
Let $\Omega \subset \mathbb{R}^3$ be a bounded Lipschitz domain. For any $s_1, s_2 \ge 0$, $1 \le p_1, p_2 \le \infty$, $\theta \in (0, 1)$ with
\begin{equation*}
s=\theta s_1+(1-\theta) s_2,\quad  \frac{1}{p}=\frac{\theta}{p_1}+\frac{1-\theta}{p_2},
\end{equation*} 
the Gagliardo--Nirenberg inequality holds:
\begin{equation*} 
\norm{f}_{W^{s,p}(\Omega)}\le C\norm{f}_{W^{s_1,p_1}(\Omega)}^\theta\norm{f}_{W^{s_2,p_2}(\Omega)}^{1-\theta},
\end{equation*}
provided either
\begin{enumerate}[label={\rm (\arabic*)}]
\item $s_1 \neq s_2$ and $\nexists$ integer $k$ such that $s_2 = k$ and $p_2 = 1$ with $s_1 - 1/p_1 \ge s_2 - 1$, or
\item $s_1 = s_2$ (reducing to H\"older's inequality).
\end{enumerate}
\end{lemma}  

The following Kato--Ponce type inequality can be found in \cite{Cheng_2016}, \cite{Kato1988}, and \cite[Proposition 2.10]{Julin2024}.
\begin{lemma}\label{leKatoPonce}
Let $\Omega\subset\mathbb{R}^3$ be bounded with $C^{1,\alpha}$ boundary ($\alpha > 0$) and $\norm{\sff}_{H^{\frac 32}(\parOmega)}\le C$. Let $k\in \frac{1}{2}\mathbb{N},l\in\mathbb{N}$, with $k,l \le 3$ and let $p_1, p_2, q_1, q_2 \in[2, \infty]$ with $p_1, q_2<\infty$ satisfying  
\begin{equation*}
\frac{1}{p_1}+\frac{1}{q_1}=\frac{1}{p_2}+\frac{1}{q_2}=\frac{1}{2}.
\end{equation*} 
Then, the following product estimates hold:
\begin{align*}
\norm{fg}_{H^k(\parOmega)} &\le C\Paren{\norm{f}_{H^k(\parOmega)}\norm{g}_{L^{\infty}(\parOmega)}+\norm{f}_{L^{\infty}(\parOmega)}\norm{g}_{H^k(\parOmega)}},\\
\norm{fg}_{H^l(\parOmega)} &\le C\Paren{\norm{f}_{W^{l, p_1}(\parOmega)}\norm{g}_{L^{q_1}(\parOmega)}+\norm{f}_{L^{p_2}(\parOmega)}\norm{g}_{W^{l, q_2}(\parOmega)}},\\
\norm{fg}_{H^k(\Omega)}&\le C	\Paren{\norm{f}_{W^{k,p_1}(\Omega)}\norm{g}_{L^{q_1}(\Omega)}+	\norm{f}_{L^{p_2}(\Omega)}\norm{g}_{W^{k,q_2}(\Omega)}}.
\end{align*}   
\end{lemma}

We will also apply the following bilinear inequality, which can be found in \cite[Lemma 2.5]{Beale1981}:
\begin{equation}\label{eqbilinearineq}
\norm{fg}_{H^s(\Omega)}\le C \norm{f}_{H^r(\Omega)} \norm{g}_{H^s(\Omega)}, \quad r > \tfrac{3}{2}, \ r \ge s \ge 0.
\end{equation}

The primary energy functional captures the highest-order part of the estimates:
\begin{equation}\label{eqenergy1}
\energy(t)\coloneqq\frac{1}{2}\Paren{\int_{\Omega_t}\module{\DT^{2}v}^2dx+\int_{\parOmega_t}\module{\Bdnabla\Paren{\DT v\cdot \normal}}^2dS+\int_{\Omega_t}\module{\nabla^{2}\Paren{\vorticity v}}^2dx}.
\end{equation} 
The composite energy functional integrates spatial regularity: 
\begin{equation}\label{eqenergy2}
E(t)\coloneqq\norm{\DT^2v}_{L^{2}(\Omega_t)}^2+\norm{\DT v}_{H^{\frac 32}(\Omega_t)}^2+\norm{v}_{H^{3}(\Omega_t)}^2+\norm{\DT v\cdot\normal}_{H^1(\parOmega_t)}^2+1.
\end{equation}  

\begin{remark}\label{rem3/2}
The $3/2$-growth in the Sobolev regularity in \eqref{eqenergy2}, as identified in \cite{Julin2024,Hao2025,Shatah2008a}, explains the energy level used above: a material derivative is effectively comparable to a $3/2$-order spatial derivative, capturing the intrinsic smoothing effect induced by surface tension.  The energy functional is constructed using the material derivative without separating the time derivative, a formulation particularly suited for energy estimates in the Eulerian framework: the local-in-time a priori estimate for system \eqref{eqEuler} is obtained by applying energy functional \eqref{eqenergy2} in \cite{Julin2024} by setting the electric field there to zero.  This type of functional leverages the structure of system \eqref{eqEuler}, allowing for a gain of $1/2$-order spatial derivative through direct substitution from \eqref{eqEuler1}; see \cite{Hao2025} for detailed discussions.
\end{remark} 

We compute the time derivative of the energy functional \eqref{eqenergy1}. 
\begin{lemma}
Under the consolidated estimates \eqref{eqCCC}, for any $t \in (0, \Tstar)$, it holds
\begin{align}
\frac{d\energy}{dt}\le{}& C(\CC)\Paren{\norm{\nabla v}_{L^\infty(\Omega_t)}+\norm{\nabla p}_{H^1(\Omega_t)}}E(t)-\int_{\Omega_t}[\DT^{2},\nabla] p\cdot \DT^{2}vdx\nonumber\\
&+\int_{\Omega_t} \DT^{2} p\divergence \DT^{2}vdx-\int_{\parOmega_t}\error (\DT^{2}v \cdot \normal)dS,\label{eqclaim1}
\end{align}
where the error $\error$ is defined in Lemma \ref{le5}.
\end{lemma}
\begin{proof} 
We define
\begin{align*} 
\energy(t)&=\frac{1}{2}\int_{\Omega_t}\module{\DT^{2}v}^2dx+\frac{1}{2}\int_{\parOmega_t}\module{\Bdnabla(\DT v\cdot \normal)}^2dS+\frac{1}{2}\int_{\Omega_t}\module{\nabla^{2} \Paren{\vorticity v}}^2dx\\
&=:\Lambda_1(t)+\Lambda_2(t)+\Lambda_3(t),
\end{align*}
and we will apply the Reynolds transport theorem (Lemma \ref{leReynolds}) several times without further mention.

From \eqref{eqEuler1} and the divergence theorem, we obtain
\begin{align*}
\frac{d\Lambda_1}{dt}
={}&\int_{\Omega_t}\DT^{3}v\cdot \DT^{2}vdx\\
={}&-\int_{\Omega_t}\DT^{2}\nabla p\cdot \DT^{2}vdx\\
={}&-\int_{\Omega_t}\nabla\DT^{2} p\cdot \DT^{2}vdx-\int_{\Omega_t}[\DT^{2},\nabla] p\cdot \DT^{2}vdx\\ 
={}&\underbrace{-\int_{\parOmega_t} \DT^{2} p (\DT^{2}v \cdot \normal) dS}_{\eqqcolon J_{1}(t)}+\int_{\Omega_t} \DT^{2} p\divergence \DT^{2}vdx-\int_{\Omega_t}[\DT^{2},\nabla] p\cdot \DT^{2}vdx.
\end{align*}  

To control the boundary energy $\Lambda_2$, we apply the commutator $[\DT, \Bdnabla]$ from Lemma \ref{le1} along with the following divergence theorem, namely,
\begin{equation*}
\int_{\parOmega_t}\Bdnabla f\cdot\Bdnabla gdS=-\int_{\parOmega_t}\LB f gdS,
\end{equation*}
to deduce
\begin{align*}
\frac{d\Lambda_2}{dt}
={}&\int_{\parOmega_t}\DT\Bdnabla(\DT v\cdot \normal)\cdot \Bdnabla(\DT v\cdot \normal)dS
+\frac{1}{2}\int_{\parOmega_t}\module{\Bdnabla(\DT v\cdot \normal)}^2\Bddiv vdS\\
={}&\int_{\parOmega_t}[\DT,\Bdnabla](\DT v\cdot \normal)\cdot \Bdnabla(\DT v\cdot \normal)dS+\int_{\parOmega_t}\Bdnabla\DT(\DT v\cdot \normal)\cdot \Bdnabla(\DT v\cdot \normal)dS\\
&+\frac{1}{2}\int_{\parOmega_t}\module{\Bdnabla(\DT v\cdot \normal)}^2\Bddiv vdS\\
={}&-\int_{\parOmega_t}\Bracket{(\Bdnabla v)^\top \Bdnabla(\DT v\cdot \normal)}\cdot \Bdnabla(\DT v\cdot \normal)dS
+\frac{1}{2}\int_{\parOmega_t}\module{\Bdnabla(\DT v\cdot \normal)}^2\Bddiv vdS\\
&+\int_{\parOmega_t}\Bdnabla(\DT^{2}v\cdot \normal)\cdot \Bdnabla(\DT v\cdot \normal)dS+\int_{\parOmega_t}\Bdnabla(\DT v\cdot \DT\normal)\cdot \Bdnabla(\DT v\cdot \normal)dS\\
\le{}&-\int_{\parOmega_t}(\DT^{2}v\cdot \normal)\LB(\DT v\cdot \normal)dS+C\norm{\Bdnabla v}_{L^\infty(\parOmega_t)}\norm{\Bdnabla(\DT v\cdot \normal)}_{L^2(\parOmega_t)}^2\\
&+\underbrace{\int_{\parOmega_t}\Bdnabla(\DT v\cdot \DT\normal)\cdot \Bdnabla(\DT v\cdot \normal)dS}_{\eqqcolon W(t)}.
\end{align*}
Again, using the formula for $\DT \normal$ in Lemma \ref{le1}, the trace theorem, and the regularity of the normal vector in \eqref{eqCCC}, we have
\begin{align*}
\module{W(t)}
\le{}&\int_{\parOmega_t}\module{\Bdnabla\DT v\tencon\DT\normal\tencon \Bdnabla(\DT v\cdot \normal)}dS+\int_{\parOmega_t}\module{\DT v\tencon \Bdnabla\DT\normal\tencon\Bdnabla(\DT v\cdot \normal)}dS\\
\le{}&C\norm{\Bdnabla\DT v}_{L^2(\parOmega_t)}\norm{\DT \normal}_{L^\infty(\parOmega_t)}\norm{\Bdnabla(\DT v\cdot \normal)}_{L^2(\parOmega_t)}\\
&+C\Paren{\norm{\DT v\tencon \Bdnabla^2v\tencon\normal}_{L^2(\parOmega_t)}+\norm{\DT v\tencon \Bdnabla v\tencon \Bdnabla \normal}_{L^2(\parOmega_t)}}\norm{\Bdnabla(\DT v\cdot \normal)}_{L^2(\parOmega_t)}\\
\le{}&C(\CC)\norm{\Bdnabla v}_{L^\infty(\parOmega_t)}\Paren{\norm{\Bdnabla\DT v}_{L^2(\parOmega_t)}^2+\norm{\Bdnabla(\DT v\cdot \normal)}_{L^2(\parOmega_t)}^2}\\
&+C(\CC)\norm{\nabla p}_{H^1(\Omega_t)}\norm{v}_{H^{3}(\Omega_t)}\norm{\Bdnabla(\DT v\cdot \normal)}_{L^2(\parOmega_t)}\\
\le{}&C(\CC)\norm{\Bdnabla v}_{L^\infty(\parOmega_t)}\Paren{\norm{\DT v}_{H^{\frac 32}(\Omega_t)}^2+\norm{\Bdnabla(\DT v\cdot \normal)}_{L^2(\parOmega_t)}^2}\\
&+C(\CC)\norm{\nabla p}_{H^1(\Omega_t)}\Paren{\norm{v}_{H^{3}(\Omega_t)}^2+\norm{\Bdnabla(\DT v\cdot \normal)}_{L^2(\parOmega_t)}^2}\\
\le{}&C(\CC)\Paren{\norm{\Bdnabla v}_{L^\infty(\parOmega_t)}+\norm{\nabla p}_{H^1(\Omega_t)}}E(t),
\end{align*}
where we have used the following results in the third inequality:
\begin{align*}
\norm{\DT v\tencon \Bdnabla^2v\tencon\normal}_{L^2(\parOmega_t)}&\le \norm{\DT v}_{L^4(\parOmega_t)} \norm{\Bdnabla^2v}_{L^4(\parOmega_t)}\norm{\normal}_{L^\infty(\parOmega_t)}\\
&\le C \norm{\DT v}_{H^{\frac 12}(\parOmega_t)} \norm{v}_{H^{\frac 52}(\parOmega_t)}\\
&\le C\norm{\nabla p}_{H^1(\Omega_t)}\norm{v}_{H^{3}(\Omega_t)},\\
\norm{\DT v\tencon \Bdnabla v\tencon \Bdnabla \normal}_{L^2(\parOmega_t)}&\le \norm{\DT v}_{L^4(\parOmega_t)} \norm{\Bdnabla v}_{L^{\infty}(\parOmega_t)}\norm{\Bdnabla \normal}_{L^4(\parOmega_t)}\\
&\le C\norm{\DT v}_{H^{\frac 12}(\parOmega_t)} \norm{v}_{H^{\frac 52}(\parOmega_t)}\norm{\normal}_{H^{\frac 32}(\parOmega_t)}\\
&\le C(\CC)\norm{\nabla p}_{H^1(\Omega_t)}\norm{v}_{H^{3}(\Omega_t)},
\end{align*}
which can be proved using the Sobolev embedding theorem, the trace theorem, \eqref{eqEuler1}, and the regularity of the normal vector in \eqref{eqCCC}. Since $H^3(\Omega_t) \hookrightarrow C^{1,\frac{1}{2}-\delta}(\overline{\Omega_t})$ in 3D (Sobolev embedding with $0<\delta\ll 1$), this guarantees the continuity of $\nabla v$ up to the boundary. Consequently, the inequality 
\begin{equation}\label{eq_Bdleint}
\norm{\Bdnabla v}_{L^\infty(\parOmega_t)} \le \norm{\nabla v}_{L^\infty(\Omega_t)}
\end{equation} 
holds by applying \cite[Lemma B.2]{Luo2024}.

Therefore, we obtain
\begin{equation*}
\frac{d\Lambda_2}{dt}\le\underbrace{-\int_{\parOmega_t}(\DT^{2}v\cdot \normal) \LB(\DT v\cdot \normal)dS}_{\eqqcolon J_2(t)}+C(\CC)\Paren{\norm{\nabla v}_{L^\infty(\Omega_t)}+\norm{\nabla p}_{H^1(\Omega_t)}}E(t).
\end{equation*}

To compute the last term involving the curl, $\Lambda_3$, we utilize Lemma \ref{le3} to obtain
\begin{align*}
\frac{d\Lambda_3}{dt}={}&\int_{\Omega_t}\DT \nabla^2 \Paren{\vorticity v}\tencon \nabla^2 \Paren{\vorticity v}dx\\
={}&\int_{\Omega_t} \nabla v \tencon \nabla^{2} \Paren{\vorticity v}\tencon \nabla^{2} \Paren{\vorticity v}+\nabla^{3} v \tencon  \Paren{\vorticity v}\tencon \nabla^{2} \Paren{\vorticity v}\\
&\qquad+ \nabla^2 v \tencon \nabla \Paren{\vorticity v}\tencon \nabla^{2} \Paren{\vorticity v} dx\\
\le{}& C \Paren{\norm{\nabla v}_{L^\infty(\Omega_t)}\norm{\nabla^{2}v}_{H^1(\Omega_t)}^2+\norm{\nabla^2 v \tencon \nabla \Paren{\vorticity v}\tencon \nabla^{2} \Paren{\vorticity v}}_{L^1(\Omega_t)}}\\
\le{}& C \Paren{\norm{\nabla v}_{L^\infty(\Omega_t)}\norm{v}_{H^3(\Omega_t)}^2+\norm{\module{\nabla^2v}^2}_{L^2(\Omega_t)}\norm{\nabla^3v}_{L^2(\Omega_t)}}. 
\end{align*}
Note that we have
\begin{equation*}
\norm{\module{\nabla^2v}^2}_{L^2(\Omega_t)}=\norm{\nabla^2 v}_{L^4(\Omega_t)}^2\le C\norm{\nabla v}_{W^{1,4}(\Omega_t)}^2
\end{equation*}
and by Gagliardo--Nirenberg inequality in Lemma \ref{leGN ineq}
\begin{equation}\label{eqnablavW14}
\norm{\nabla v}_{W^{1,4}(\Omega_t)}\le C\norm{\nabla v}_{H^2(\Omega_t)}^{\frac 12}\norm{\nabla v}_{L^\infty(\Omega_t)}^{\frac 12}.
\end{equation}  
Thus, it holds
\begin{equation*}
\norm{\module{\nabla^2v}^2}_{L^2(\Omega_t)}\norm{\nabla^3v}_{L^2(\Omega_t)}\le C\norm{\nabla v}_{L^\infty(\Omega_t)}\norm{v}_{H^3(\Omega_t)}^2
\end{equation*}
and 
\begin{equation*}
\frac{d\Lambda_3}{dt}
\le C\norm{\nabla v}_{L^\infty(\Omega_t)}E(t).
\end{equation*}

Combining with the above calculations and applying Lemma \ref{le5}, we obtain
\begin{align*}
J_1(t)+J_2(t)&=-\int_{\parOmega_t} \DT^{2} p (\DT^{2}v \cdot \normal) dS-\int_{\parOmega_t}(\DT^{2}v\cdot \normal) \LB(\DT v\cdot \normal)dS\\
&=-\int_{\parOmega_t}\error (\DT^{2}v \cdot \normal)dS.
\end{align*}
Therefore, the claim \eqref{eqclaim1} follows.
\end{proof}

The error term $\error$ can be estimated in the following lemma.
\begin{lemma} 
Under \eqref{eqCCC}, for any $t \in (0, \Tstar)$, we have
\begin{equation}\label{eqRp}
\norm{\error}_{H^{\frac 12}(\parOmega_t)}\le C(\CC)\Paren{  \norm{\nabla v}_{L^\infty(\Omega_t)}+\norm{\nabla p}_{H^1(\Omega_t)}+1} \sqrt{E(t)}.
\end{equation}
\end{lemma}
\begin{proof}
From Lemma \ref{le5}, the error $\error$ can be expressed as
\begin{equation*}
\error 
=-\module{\sff}^2 \DT v\cdot\normal+\Bdnabla p\cdot\DT v +\Bdnabla^2v\tencon\Bdnabla v\tencon\normal+\Bdnabla v\tencon\Bdnabla v\tencon \sff.
\end{equation*}

For the second term, recalling the curvature bound provided in \eqref{eqCCC}, we proceed by applying Kato--Ponce estimates from Lemma \ref{leKatoPonce}. It follows from \eqref{eqEuler1} and the definition of the tangential derivative in \eqref{eqtangentialderivative} that 
\begin{equation*}
\Bdnabla p\cdot\DT v=-\Bdnabla p\cdot\nabla p=-\Bdnabla p\cdot\Paren{\Bdnabla p+(\nabla p \cdot\normal)\normal}=-\Bdnabla p\cdot \Bdnabla p,
\end{equation*} 
since $\Bdnabla p\cdot\normal$ vanishes. Next, by the extension of the normal vector satisfying $\norm{\normal}_{H^{3}(\Omega_t)} \le C(\CC)$ as given in \eqref{eqextendnormal}, and by the trace theorem, we obtain
\begin{equation*}
\norm{\Bdnabla p\cdot\DT v}_{H^{\frac 12}(\parOmega_t)}=\norm{\Bdnabla p\cdot\Bdnabla p}_{H^{\frac 12}(\parOmega_t)}\le C\norm{\Bdnabla p\cdot\Bdnabla p}_{H^1(\Omega_t)}\le C(\CC)\norm{\module{\nabla p}^2}_{H^1(\Omega_t)}. 
\end{equation*}
We then apply \eqref{eqEuler1}, Lemma \ref{leKatoPonce}, and the Sobolev embedding theorem to obtain
\begin{align*}
\norm{\module{\nabla p}^2}_{H^1(\Omega_t)}
\le{}& C(\CC)\norm{\nabla p}_{L^{6}(\Omega_t)}\norm{\nabla p}_{W^{1,3}(\Omega_t)}\\
\le{}& C(\CC)\norm{\nabla p}_{H^1(\Omega_t)}\norm{\nabla p}_{H^{\frac 32}(\Omega_t)}\\
\le{}& C(\CC)\norm{\nabla p}_{H^1(\Omega_t)}\sqrt{E(t)}.
\end{align*}

For the first term, again utilizing the curvature bound in \eqref{eqCCC}, the bilinear inequality \eqref{eqbilinearineq}, and the trace theorem, we have
\begin{equation*}
\norm{\module{\sff}^2\DT v\cdot\normal}_{H^{\frac 12}(\parOmega_t)}\le C(\CC)\norm{\sff}_{H^{\frac 32}(\parOmega_t)}^2\norm{\DT v\cdot\normal}_{H^{\frac 12}(\parOmega_t)}\le C(\CC)\sqrt{E(t)}.
\end{equation*} 

For the remaining terms, owing to the regularity of the normal vector and the second fundamental form as stated in \eqref{eqCCC}, and by the bilinear inequality \eqref{eqbilinearineq}, it suffices to show that  
\begin{equation*}
\norm{\Bdnabla^2v\tencon\Bdnabla v}_{H^{\frac 12}(\parOmega_t)} 
\le C(\CC)\norm{\nabla v}_{L^\infty(\Omega_t)}\sqrt{E(t)}.
\end{equation*}
Indeed, by applying \eqref{eqH1/2}, one obtains
\begin{equation}\label{eqnununu} 
\norm{\Bdnabla^2v\tencon\Bdnabla v}_{H^{\frac 12}(\parOmega_t)} 
\le  C(\CC)\Paren{\norm{\Bdnabla^2v\tencon\Bdnabla v}_{L^{2}(\parOmega_t)}+\norm{\nabla\Paren{\Bdnabla^2v\tencon\Bdnabla v}}_{L^2(\Omega_t)}}. 
\end{equation} 

To control the second term, a straightforward calculation reveals that the following tensors can be expressed in terms of the $\tencon$-products:
\begin{align*}
\Bdnabla v={}&\nabla v+\nabla v\tencon\normal^{\tencon,2},\\
\nabla \Bdnabla v
={}&\nabla^2 v+\nabla^2 v\tencon\normal^{\tencon,2}+\nabla v\tencon\nabla \normal\tencon\normal,\\
\Bdnabla^2v
={}&\nabla^2 v+\nabla^2 v\tencon\normal^{\tencon,2}+\nabla^2 v\tencon\normal^{\tencon,4}+\nabla v\tencon\nabla \normal\tencon\normal+\nabla v\tencon\nabla \normal\tencon\normal^{\tencon,3},\\
\nabla\Bdnabla^2v={}&\nabla^3 v+\nabla^3 v\tencon\normal^{\tencon,2}+\nabla^3 v\tencon\normal^{\tencon,4}+\nabla^2 v\tencon\nabla\normal\tencon\normal^{\tencon,3}+\nabla^2 v\tencon\nabla\normal\tencon\normal\\
&+\nabla v\tencon\nabla^2 \normal\tencon\normal+\nabla v\tencon\nabla^2 \normal\tencon\normal^{\tencon,3}+\nabla v\tencon\Paren{\nabla \normal}^{\tencon,2}+\nabla v\tencon\Paren{\nabla \normal}^{\tencon,2}\tencon\normal^{\tencon,2},
\end{align*}
where $T^{\tencon,m}$ denotes the $m$-fold $\tencon$-product of the tensor $T$. 

Then, we proceed to estimate 
\begin{equation*}
\norm{\nabla\Paren{\Bdnabla^2v\tencon\Bdnabla v}}_{L^2(\Omega_t)}\le C\Paren{ \underbrace{\norm{\nabla\Bdnabla^2v\tencon\Bdnabla v}_{L^2(\Omega_t)}}_{\eqqcolon \Pi_1}+\underbrace{\norm{\Bdnabla^2v\tencon\nabla\Bdnabla v}_{L^2(\Omega_t)}}_{\eqqcolon \Pi_2}}.
\end{equation*}

To estimate $\Pi_1$, for sufficiently small $\varepsilon > 0$, by the bilinear inequality \eqref{eqbilinearineq} and \eqref{eqextendnormal}, we obtain
\begin{equation*}
\norm{\Paren{\nabla v\tencon\normal^{\tencon,2}}\tencon\Paren{\nabla^3 v\tencon\normal^{\tencon,4}}}_{L^2(\Omega_t)}\le C(\CC)\norm{\nabla v}_{L^\infty(\Omega_t)}\norm{v}_{H^3(\Omega_t)}.
\end{equation*}
Similarly, for sufficiently small $\varepsilon > 0$, we obtain
\begin{align*} 
\norm{\Paren{\nabla v\tencon\normal^{\tencon,2}}\tencon\Paren{\nabla^2 v\tencon\nabla\normal\tencon\normal^{\tencon,3}}}_{L^2(\Omega_t)}
&\le  C\norm{\normal}_{H^{\frac 32+\varepsilon}(\Omega_t)}^5\norm{\normal}_{H^{\frac 52+\varepsilon}(\Omega_t)}\norm{\nabla v}_{L^\infty(\Omega_t)}\norm{v}_{H^3(\Omega_t)}\\
&\le C(\CC)\norm{\nabla v}_{L^\infty(\Omega_t)}\norm{v}_{H^3(\Omega_t)},\\
\norm{\Paren{\nabla v\tencon\normal^{\tencon,2}}\tencon\Paren{\nabla v\tencon\nabla^2\normal\tencon\normal^{\tencon,3}}}_{L^2(\Omega_t)}&\le C\norm{\normal}_{H^{\frac 32+\varepsilon}(\Omega_t)}^5\norm{\nabla^2\normal}_{L^{3}(\Omega_t)}\norm{\nabla v}_{L^\infty(\Omega_t)}\norm{v}_{H^3(\Omega_t)}\\
&\le C(\CC)\norm{\nabla v}_{L^\infty(\Omega_t)}\norm{v}_{H^3(\Omega_t)},\\
\norm{\Paren{\nabla v\tencon\normal^{\tencon,2}}\tencon\Paren{\nabla v\tencon\Bracket{\nabla \normal}^{\tencon,2}\tencon\normal^{\tencon,2}}}_{L^2(\Omega_t)}&\le C\norm{\normal}_{H^{\frac 32+\varepsilon}(\Omega_t)}^4\norm{\nabla\normal}_{H^{\frac 32+\varepsilon}(\Omega_t)}^2\norm{\nabla v}_{L^\infty(\Omega_t)}\norm{v}_{H^3(\Omega_t)}\\
&\le C(\CC)\norm{\nabla v}_{L^\infty(\Omega_t)}\norm{v}_{H^3(\Omega_t)},
\end{align*}
and the remaining estimates in $\Pi_1$ are straightforward.

To handle $\Pi_2$, it suffices to control the product term
\begin{equation*} 
\Paren{\nabla^2 v\tencon\normal^{\tencon,2}}\tencon\Paren{\nabla^2 v\tencon\normal^{\tencon,4}},
\end{equation*} 
since the remaining terms are either straightforward or have already been computed above. By Lemma \ref{leGN ineq} and \eqref{eqextendnormal}, it follows that
\begin{equation*}
\norm{\Paren{\nabla^2 v\tencon\normal^{\tencon,2}}\tencon\Paren{\nabla^2 v\tencon\normal^{\tencon,4}}}_{L^2(\Omega_t)}\le C(\CC) \norm{\nabla^2 v}_{L^4(\Omega_t)}^2 \le C(\CC) \norm{\nabla v}_{L^\infty(\Omega_t)}\norm{\nabla v}_{H^{2}(\Omega_t)}.
\end{equation*}
We conclude that
\begin{equation*}
\norm{\nabla\Paren{\Bdnabla^2v\tencon\Bdnabla v}}_{L^2(\Omega_t)}\le C(\CC)\norm{\nabla v}_{L^\infty(\Omega_t)}\norm{v}_{H^3(\Omega_t)},
\end{equation*}
and thus, from \eqref{eqnununu}, the trace theorem, and \eqref{eq_Bdleint}, it follows that
\begin{align*} 
\norm{\Bdnabla^2v\tencon\Bdnabla v}_{H^{\frac 12}(\parOmega_t)}  
\le{}&C(\CC) \Paren{\norm{\nabla v}_{L^\infty(\Omega_t)} \norm{\Bdnabla^2v}_{L^{2}(\parOmega_t)}+\norm{\nabla v}_{L^\infty(\Omega_t)}\norm{v}_{H^3(\Omega_t)}}\\ 
\le{}& C(\CC)\norm{\nabla v}_{L^\infty(\Omega_t)}\sqrt{E(t)}.
\end{align*}

The estimate for $\norm{\Bdnabla v\tencon\Bdnabla v}_{H^{\frac 12}(\parOmega_t)}$ can be derived similarly, and we omit the details. Therefore, \eqref{eqRp} follows. 
\end{proof}

Then, the terms $\int_{\parOmega_t}\error (\DT^{2}v \cdot \normal)dS$ and $\int_{\Omega_t}[\DT^{2},\nabla] p\cdot \DT^{2}vdx$ in \eqref{eqclaim1} can be controlled as follows.
\begin{lemma}
Under \eqref{eqCCC}, for any $t \in (0, \Tstar)$, it holds
\begin{equation}\label{eqclaim2}
\module{\int_{\parOmega_t}\error (\DT^{2}v \cdot \normal)dS}+\module{\int_{\Omega_t}[\DT^{2},\nabla] p\cdot \DT^{2}vdx}\le C(\CC)\Paren{\norm{\nabla v}_{L^\infty(\Omega_t)}+\norm{\nabla p}_{H^1(\Omega_t)}+1}E(t).
\end{equation}
\end{lemma}
\begin{proof} 
For the first term, it holds
\begin{equation*}
\module{\int_{\parOmega_t}\error (\DT^{2}v \cdot \normal)dS}
\le C\norm{\DT^{2}v\cdot\normal}_{H^{-\frac{1}{2}}(\parOmega_t)}\norm{\error}_{H^{\frac{1}{2}}(\parOmega_t)}.
\end{equation*}
Note that, from \eqref{eqCCC}, the free boundary $\parOmega_t \in H^{\frac 72}$, and thus, we can apply the normal trace theorem (see, e.g., \cite[Theorem 3.1]{Cheng2007} and \cite[Lemma 5.1]{Coutand2010}). 
\begin{equation*}
\norm{\DT^{2}v\cdot\normal}_{H^{-\frac{1}{2}}(\parOmega_t)}\le C(\CC)\Paren{\norm{\DT^2v}_{L^{2}(\Omega_t)}+\norm{\divergence \DT^{2}v}_{H^{-1}(\Omega_t)}},
\end{equation*}
and we have
\begin{align}
\norm{\divergence \DT^{2}v}_{H^{-1}(\Omega_t)}\le{}&\sup \Brace{\module{\int_{\Omega_t}\divergence \DT^{2}v Fdx}:F\in H^1_0(\Omega_t),\norm{F}_{H^1_0(\Omega_t)}\le 1}\nonumber\\
\le{}&\sup \Brace{\module{\int_{\Omega_t}\DT^{2}v\cdot \nabla Fdx}:F\in H^1_0(\Omega_t),\norm{F}_{H^1_0(\Omega_t)}\le 1}\nonumber\\
\le{}& \norm{\DT^{2}v}_{L^{2}(\Omega_t)}.\label{eqH-1}
\end{align} 
We then use \eqref{eqRp} to obtain
\begin{equation*}
\module{\int_{\parOmega_t}\error (\DT^{2}v \cdot \normal)dS}\le C(\CC)\Paren{\norm{\nabla v}_{L^\infty(\Omega_t)}+\norm{\nabla p}_{H^1(\Omega_t)}+1}E(t).
\end{equation*} 

For the second term, by applying \eqref{eqEuler1}, \eqref{eqDtlnablap}, and the Sobolev embedding theorem, we obtain
\begin{align*}
\module{\int_{\Omega_t}[\DT^{2},\nabla] p\cdot \DT^{2}vdx}
\le{}& C\norm{[\DT^{2},\nabla] p}_{L^2(\Omega_t)}\norm{\DT^2v}_{L^2(\Omega_t)}\\
\le{}& C\Paren{\norm{\nabla v\tencon \DT^2 v}_{L^2(\Omega_t)}+\norm{\nabla \DT v\tencon \DT v}_{L^2(\Omega_t)}}\norm{\DT^2v}_{L^2(\Omega_t)}\\
\le{}& C\left(\norm{\nabla v}_{L^\infty(\Omega_t)}\norm{\DT^2v}_{L^2(\Omega_t)}\right.\\
&\quad\left.+\norm{\nabla \DT v}_{L^3(\Omega_t)}\norm{\DT v}_{L^6(\Omega_t)}\right)\norm{\DT^2v}_{L^2(\Omega_t)}\\
\le{}& C\Paren{\norm{\nabla v}_{L^\infty(\Omega_t)}+\norm{\nabla p}_{H^1(\Omega_t)}}\Paren{\norm{\DT^2v}_{L^2(\Omega_t)}^2+\norm{\DT v}_{H^{\frac 32}(\Omega_t)}^2}\\
\le{}& C\Paren{\norm{\nabla v}_{L^\infty(\Omega_t)}+\norm{\nabla p}_{H^1(\Omega_t)}}E(t).
\end{align*}  
\end{proof}

Combining the above three lemmas, we obtain
\begin{equation}\label{eqmidstep}
\frac{d}{dt}\energy(t)\le C(\CC)\Paren{\norm{\nabla v}_{L^\infty(\Omega_t)}+\norm{\nabla p}_{H^1(\Omega_t)}+1}E(t)+I(t),
\end{equation}
where
\begin{equation}\label{eqdefI(t)}
I(t)=\int_{\Omega_t} \DT^{2} p\divergence \DT^{2}vdx.
\end{equation} 
To estimate $I(t)$, we make use of the following result.
\begin{lemma}
Under \eqref{eqCCC}, we have
\begin{equation}\label{eqclaim3}
\norm{\divergence \DT^{2}v}_{H^{\frac{1}{2}}(\Omega_t)}\le C(\CC)\Paren{\norm{\nabla v}_{L^{\infty}(\Omega_t)}+\norm{\nabla^2p}_{L^{2}(\Omega_t)}}\sqrt{E(t)},\quad t \in (0, \Tstar).
\end{equation}
\end{lemma}
\begin{proof}
By Lemma \ref{le3}, it follows that
\begin{equation*}
\norm{\divergence \DT^{2}v}_{H^{\frac{1}{2}}(\Omega_t)}\le C\Paren{\norm{\partial_i v^j \partial_j\DT v^i}_{H^{\frac{1}{2}}(\Omega_t)}+\norm{\partial_i v^j\partial_jv^k\partial_kv^i}_{H^{\frac{1}{2}}(\Omega_t)}}.
\end{equation*}

For the first term, by applying Lemma \ref{leKatoPonce}, we obtain
\begin{align*}
\norm{\partial_i v^j \partial_j\DT v^i}_{H^{\frac{1}{2}}(\Omega_t)}\le{}& C\Paren{\norm{\nabla v}_{L^{\infty}(\Omega_t)}\norm{\nabla \DT v}_{H^{\frac{1}{2}}(\Omega_t)}+\norm{\nabla\DT v}_{L^{\frac{12}{5}}(\Omega_t)}\norm{\nabla v}_{W^{\frac{1}{2},12}(\Omega_t)}}\\
\le{}& C\Paren{\norm{\nabla v}_{L^{\infty}(\Omega_t)}\norm{\DT v}_{H^{\frac{3}{2}}(\Omega_t)}+\norm{\nabla^2p}_{L^{2}(\Omega_t)}^{\frac 12}\norm{\nabla v}_{L^{\infty}(\Omega_t)}^{\frac 12}\sqrt{E(t)}}\\
\le{}& C\Paren{\norm{\nabla v}_{L^{\infty}(\Omega_t)}+\norm{\nabla^2p}_{L^{2}(\Omega_t)}}\sqrt{E(t)},
\end{align*}
where we have used the following estimates
\begin{align*}
\norm{\nabla\DT v}_{L^{\frac{12}{5}}(\Omega_t)}&\le C\norm{\nabla\DT v}_{H^{\frac{1}{4}}(\Omega_t)}\le C\norm{\DT v}_{H^{\frac{5}{4}}(\Omega_t)}\le C\norm{\nabla p}_{H^{1}(\Omega_t)}^{\frac 12}\norm{\DT v}_{H^{\frac{3}{2}}(\Omega_t)}^{\frac 12},\\
\norm{\nabla v}_{W^{\frac{1}{2},12}(\Omega_t)}&\le C\norm{\nabla v}_{W^{1,6}(\Omega_t)}^{\frac 12}\norm{\nabla v}_{L^{\infty}(\Omega_t)}^{\frac 12}\le C\norm{\nabla v}_{L^{\infty}(\Omega_t)}^{\frac 12}\norm{v}_{H^{3}(\Omega_t)}^{\frac 12},
\end{align*}
which can be derived by applying Lemma \ref{leGN ineq}, \eqref{eqEuler1}, and the Sobolev embedding theorem.

To handle the second term, recalling from \eqref{eqCCC} that the $L^2$-norms of the velocity and vorticity are bounded by $C(\CC)$, and combining these with  divergence-free condition \eqref{eqEuler2} and the lower-order div-curl estimate \eqref{eqdivcurlk} for $k = 1$, we obtain the following bound
\begin{equation}\label{eqvH1}
\norm{v}_{H^1(\Omega_t)}\le C(\CC)\Paren{\norm{\divergence v}_{L^2(\Omega_t)}+\norm{\vorticity v}_{L^2(\Omega_t)}+\norm{v_n}_{H^{\frac 12}(\parOmega_t)}+\norm{v}_{L^2(\Omega_t)}}\le C(\CC).
\end{equation}
We then apply Lemma \ref{leKatoPonce} to deduce that \begin{align}
\norm{\module{\nabla v}^2}_{H^{\frac 12}(\Omega_t)}&\le C\norm{\nabla v}_{W^{\frac 12,\frac{1}{\delta}}(\Omega_t)}\norm{\nabla v}_{L^{\frac{2}{1-2\delta}}(\Omega_t)}\nonumber\\
&\le C\norm{v}_{H^{3-3\delta}(\Omega_t)}\norm{\nabla v}_{L^{\frac{2}{1-2\delta}}(\Omega_t)}\nonumber\\
&\le C(\CC)\norm{v}_{H^{3}(\Omega_t)},\label{eqnablav2}
\end{align}  
where $\delta > 0$ is sufficiently small, and we have used the following results, which are obtained by applying \eqref{eqvH1}, Sobolev embedding and interpolation 
\begin{align*}
\norm{v}_{H^{3-3\delta}(\Omega_t)}&\le C(\CC)\norm{v}_{H^{3}(\Omega_t)}^{1-\frac{3\delta}{2}}\norm{v}_{H^{1}(\Omega_t)}^{\frac{3\delta}{2}}\le C(\CC)\norm{v}_{H^{3}(\Omega_t)}^{1-\frac{3\delta}{2}},\\
\norm{\nabla v}_{L^{\frac{2}{1-2\delta}}(\Omega_t)}&\le C(\CC)\norm{v}_{H^{1+3\delta}(\Omega_t)}\le C(\CC)\norm{v}_{H^{3}(\Omega_t)}^{\frac{3\delta}{2}}\norm{v}_{H^{1}(\Omega_t)}^{1-\frac{3\delta}{2}}\le C(\CC)\norm{v}_{H^{3}(\Omega_t)}^{\frac{3\delta}{2}}. 
\end{align*}
Therefore, again from Lemma \ref{leKatoPonce}, the Sobolev embedding theorem, and \eqref{eqnablav2}, it follows that
\begin{align*}
\|\partial_i v^j\partial_jv^k\partial_kv^i\|_{H^{\frac{1}{2}}(\Omega_t)}
&\le C(\CC)\Paren{\norm{\nabla v}_{L^{\infty}(\Omega_t)}\norm{\module{\nabla v}^2}_{H^{\frac 12}(\Omega_t)}+\norm{\nabla v}_{W^{\frac 12,\frac{1}{\delta}}(\Omega_t)}\norm{\module{\nabla v}^2}_{L^{\frac{2}{1-2\delta}}(\Omega_t)} }\\
&\le C(\CC)\Paren{\norm{\nabla v}_{L^{\infty}(\Omega_t)}\norm{v}_{H^3(\Omega_t)}+\norm{\nabla v}_{L^{\infty}(\Omega_t)}\norm{v}_{H^{3-3\delta}(\Omega_t)}\norm{\nabla v}_{L^{\frac{2}{1-2\delta}}(\Omega_t)}}\\
&\le C(\CC)\norm{\nabla v}_{L^{\infty}(\Omega_t)}\sqrt{E(t)}.
\end{align*}
Therefore, \eqref{eqclaim3} follows. This completes the proof.
\end{proof}

\begin{proposition}\label{prenergy estimate} 
Under \eqref{eqCCC}, for any $t\in (0, \Tstar)$, it holds that
\begin{equation}\label{eqclaim4}
\module{I(t)}\le C(\CC)\Paren{\norm{\nabla v}_{L^\infty(\Omega_t)}+\norm{\nabla p}_{H^1(\Omega_t)}+1}  E(t), 
\end{equation}
where $I(t)$ is defined in \eqref{eqdefI(t)}, and the energy functional \eqref{eqenergy1} satisfies:
\begin{equation}\label{eqenergy estimate}
\frac{d\energy}{dt}\le C(\CC)\Paren{\norm{\nabla v}_{L^\infty(\Omega_t)}+1}E(t).
\end{equation}
\end{proposition}

\begin{proof}  
Consider the following elliptic equation 
\begin{equation}\label{eqellipticu}
\begin{cases}
-\Delta u=\divergence \DT^{2}v, & \text{ in } \Omega_t,\\
u=0, &  \text{ on } \parOmega_t.
\end{cases}
\end{equation}
We begin by applying the elliptic estimates in Lemma \ref{leellipticestimate}
\begin{equation}\label{eqellipticest}
\begin{aligned}
\norm{\partial_\normal u}_{H^1(\parOmega_t)}+\norm{\nabla u}_{H^{\frac{3}{2}}(\Omega_t)} &\le C\norm{\divergence \DT^{2}v}_{H^{\frac{1}{2}}(\Omega_t)},\\ \norm{u}_{H^1(\Omega_t)}&\le C\norm{\DT^2 v}_{L^2(\Omega_t)}.
\end{aligned}
\end{equation}  

We then integrate by parts to obtain
\begin{equation*}
I(t)=-\int_{\Omega_t}\Delta\DT^{2} p udx-\int_{\parOmega_t}\DT^{2} p\partial_\normal udS\eqqcolon I_{1}(t)+I_{2}(t).
\end{equation*} 
By applying Lemma \ref{le5} and integrating by parts once more, we have
\begin{align*}
-I_2(t) &=\int_{\parOmega_t}\Paren{-\LB (\DT v\cdot\normal)+\error}\partial_\normal udS\\
&=\int_{\parOmega_t} \Bdnabla(\DT v\cdot\normal)\cdot\Bdnabla\partial_\normal udS+\int_{\parOmega_t}\error\partial_\normal udS.
\end{align*}  
By the normal trace theorem, \eqref{eqellipticu}, \eqref{eqH-1}, and the second elliptic estimate in \eqref{eqellipticest}, we obtain
\begin{align*}
\norm{\partial_\normal u}_{H^{-\frac 12}(\parOmega_t)}&\le C\Paren{\norm{\nabla u}_{L^{2}(\Omega_t)}+\norm{\Delta u}_{H^{-1}(\Omega_t)}}\\
&\le C\Paren{\norm{\DT^2v}_{L^{2}(\Omega_t)}+\norm{\divergence \DT^{2}v}_{H^{-1}(\Omega_t)}}\\
&\le  C\norm{\DT^2v}_{L^{2}(\Omega_t)}.
\end{align*}
Then, we use the first elliptic estimate in \eqref{eqellipticest}, \eqref{eqclaim3}, and \eqref{eqRp} to deduce
\begin{align*}
\module{I_{2}(t)}\le{}& C\Paren{\norm{\Bdnabla(\DT v\cdot\normal)}_{L^2(\parOmega_t)}\norm{\partial_\normal u}_{H^1(\parOmega_t)}+\norm{\error}_{H^{\frac 12}(\parOmega_t)}\norm{\partial_\normal u}_{H^{-\frac 12}(\parOmega_t)}}\\
\le{}& C\sqrt{E(t)}\norm{\divergence \DT^{2}v}_{H^{\frac{1}{2}}(\Omega_t)}\\
&+C(\CC)\Paren{\norm{\nabla v}_{L^\infty(\Omega_t)}+\norm{\nabla p}_{H^1(\Omega_t)}+1}  \sqrt{E(t)}\norm{\DT^2v}_{L^{2}(\Omega_t)}\\
\le{}& C(\CC)\Paren{\norm{\nabla v}_{L^{\infty}(\Omega_t)}+\norm{\nabla^2p}_{L^{2}(\Omega_t)}}E(t)\\
&+C(\CC)\Paren{\norm{\nabla v}_{L^\infty(\Omega_t)}+\norm{\nabla p}_{H^1(\Omega_t)}+1} E(t) \\
\le{}&C(\CC)\Paren{\norm{\nabla v}_{L^\infty(\Omega_t)}+\norm{\nabla p}_{H^1(\Omega_t)}+1}E(t).
\end{align*}  

For $I_1(t)$, by Lemma \ref{le4}, \eqref{eqEuler1}, and  divergence-free condition \eqref{eqEuler2}, we integrate by parts to obtain
\begin{align*}
I_1(t)={}&\int_{\Omega_t}\Big[4\partial_i v^j \partial_j\DT^2 v^i+3\partial_i\DT v^j \partial_j\DT v^i-12\partial_i v^l\partial_l v^j \partial_j\DT v^i\\
&\qquad+6\partial_iv^l\partial_l v^j\partial_jv^k\partial_kv^i+\sum_j\partial_j \Paren{2\partial_j v_i\DT^2 v^i+\partial_j \DT v_i\DT v^i}\Big]udx\\
={}&\int_{\Omega_t}\Big[-4\partial_i v^j \DT^2 v^i\partial_ju-2\sum_j\partial_j v_i\DT^2 v^i\partial_j u+\sum_{i,j}\partial_j \DT v_i\partial_i p\partial_j u\\
&\qquad+3\partial_i\DT v^j \partial_j\DT v^iu-12\partial_i v^l\partial_l v^j \partial_j\DT v^iu+6\partial_iv^l\partial_l v^j\partial_jv^k\partial_kv^iu\Big]dx.
\end{align*}

We use \eqref{eqellipticest} and the Sobolev embedding theorem to estimate as follows
\begin{align*}
\norm{\partial_i v^j \DT^2 v^i\partial_ju}_{L^1(\Omega_t)}\le{}&C\norm{\nabla v}_{L^\infty(\Omega_t)}\norm{\DT^2v}_{L^2(\Omega_t)}\norm{\nabla u}_{L^2(\Omega_t)} \\
\le{}& C\norm{\nabla v}_{L^\infty(\Omega_t)}E(t),\\
\sum_j\norm{\partial_j v_i\DT^2 v^i\partial_j u}_{L^1(\Omega_t)} \le{}& C\norm{\nabla v}_{L^\infty(\Omega_t)}E(t),\\
\sum_{i,j}\norm{\partial_j \DT v_i\partial_i p\partial_j u}_{L^1(\Omega_t)}\le{}&C\norm{\nabla p}_{L^6(\Omega_t)}\norm{\nabla\DT v}_{L^3(\Omega_t)}\norm{\DT^2v}_{L^2(\Omega_t)}\\
\le{}& C\norm{\nabla p}_{H^1(\Omega_t)}E(t).
\end{align*}

For the fourth term, we apply \eqref{eqEuler1}, \eqref{eqdivDtv},  and integrate by parts to deduce
\begin{align*}
\int_{\Omega_t}\partial_i\DT v^j \partial_j\DT v^iudx
={}&-\int_{\Omega_t}\DT v^j \partial_j\divergence\DT vudx-\int_{\Omega_t}\DT v^j \partial_j\DT v^i\partial_iudx\\
={}&\int_{\Omega_t}\module{\divergence\DT v}^2udx+\int_{\Omega_t}\DT v^j \divergence\DT v\partial_judx\\
&-\int_{\Omega_t}\DT v^j \partial_j\DT v^i\partial_iudx\\
={}&\int_{\Omega_t}\partial_i v^j\partial_j v^i \partial_k v^l\partial_l v^kudx-\sum_j\int_{\Omega_t}\partial_j p \partial_i v^k\partial_k v^i\partial_judx\\
&+\sum_j\int_{\Omega_t}\partial_j p \partial_j\DT v^i\partial_iudx.
\end{align*}
Then, it follows that
\begin{align*}
\module{\int_{\Omega_t}\partial_i\DT v^j \partial_j\DT v^iudx}
\le{}&C\norm{\nabla v}_{L^\infty(\Omega_t)}\norm{\module{\nabla v}^3}_{L^{\frac{6}{5}}(\Omega_t)}\norm{u}_{L^6(\Omega_t)}\\
&+C\norm{\nabla p}_{L^6(\Omega_t)}\norm{\module{\nabla v}^2}_{L^{3}(\Omega_t)}\norm{\nabla u}_{L^2(\Omega_t)}\\
&+C\norm{\nabla p}_{L^6(\Omega_t)}\norm{\nabla \DT v}_{L^3(\Omega_t)}\norm{\nabla u}_{L^2(\Omega_t)}. 
\end{align*}
Note that by applying Lemma \ref{leGN ineq} and \eqref{eqvH1}, one obtains
\begin{align}
\norm{\module{\nabla v}^2}_{L^{3}(\Omega_t)}\le{}&\norm{v}_{W^{1,6}(\Omega_t)}^2\le C\norm{v}_{H^{1}(\Omega_t)}\norm{v}_{H^3(\Omega_t)}\le C(\CC)\norm{v}_{H^{3}(\Omega_t)},\label{e_23}\\
\norm{\module{\nabla v}^3}_{L^{\frac{6}{5}}(\Omega_t)}\le{}& \norm{v}_{W^{1,\frac{18}{5}}(\Omega_t)}^3\le C\norm{v}_{H^{3}(\Omega_t)}\norm{v}_{L^6(\Omega_t)}^2\le C(\CC)\norm{v}_{H^{3}(\Omega_t)}.\nonumber
\end{align}  
These, combined with the second elliptic estimate in \eqref{eqellipticest} and the Sobolev embedding theorem, yield that
\begin{align*}
\module{\int_{\Omega_t}\partial_i\DT v^j \partial_j\DT v^iudx}
\le{}&C(\CC)\Paren{\norm{\nabla v}_{L^\infty(\Omega_t)}+\norm{\nabla p}_{H^1(\Omega_t)}}E(t).
\end{align*}

Similarly, for the fifth term, from \eqref{e_23} and the second elliptic estimate in \eqref{eqellipticest}, it follows that
\begin{align*}
\norm{\partial_i v^l\partial_l v^j \partial_j\DT v^iu}_{L^1(\Omega_t)}\le{}&\norm{\module{\nabla v}^2}_{L^3(\Omega_t)}\norm{\nabla^2p}_{L^2(\Omega_t)}\norm{u}_{L^6(\Omega_t)}\\
\le{}& C(\CC)\norm{\nabla^2p}_{L^2(\Omega_t)}E(t),
\end{align*} 

For the last term, note that by applying integration by parts, one has
\begin{equation*}
\int_{\Omega_t}\partial_iv^l\partial_l v^j\partial_jv^k\partial_kv^iudx=-\int_{\Omega_t}v^l\partial_l v^j\partial_jv^k\partial_kv^i\partial_iudx-\int_{\Omega_t}v^l\partial_i\Paren{\partial_l v^j\partial_jv^k}\partial_kv^iudx.
\end{equation*}
Then, we apply \eqref{e_23}, \eqref{eqellipticest}, the Sobolev embedding theorem, and \eqref{eqvH1} to obtain
\begin{align*}
\module{\int_{\Omega_t}v^l\partial_l v^j\partial_jv^k\partial_kv^i\partial_iudx}\le{}&C\norm{\nabla v}_{L^\infty(\Omega_t)}\norm{v}_{L^6(\Omega_t)}\norm{\module{\nabla v}^2}_{L^3(\Omega_t)}\norm{\nabla u}_{L^2(\Omega_t)}\\
\le{}&C(\CC)\norm{\nabla v}_{L^\infty(\Omega_t)}\norm{v}_{H^3(\Omega_t)}\norm{\nabla u}_{L^2(\Omega_t)}\\
\le{}&C(\CC)\norm{\nabla v}_{L^\infty(\Omega_t)}E(t).
\end{align*}
Similarly, one has
\begin{align*} 
\module{\int_{\Omega_t}v^l\partial_i\Paren{\partial_l v^j\partial_jv^k}\partial_kv^iudx}
\le{}&C\norm{\nabla v}_{L^\infty(\Omega_t)}\norm{v}_{L^6(\Omega_t)}\norm{\nabla^2 v}_{L^6(\Omega_t)}\\
&\quad\cdot\norm{\nabla v}_{L^2(\Omega_t)}\norm{u}_{L^6(\Omega_t)}\\
\le{}&C(\CC)\norm{\nabla v}_{L^\infty(\Omega_t)}\norm{v}_{H^3(\Omega_t)}\norm{u}_{H^1(\Omega_t)}\\
\le{}&C(\CC)\norm{\nabla v}_{L^\infty(\Omega_t)}E(t).
\end{align*}
We conclude that
\begin{equation*}
I_1(t)\le C(\CC)\Paren{\norm{\nabla v}_{L^\infty(\Omega_t)}+\norm{\nabla p}_{H^1(\Omega_t)}}E(t),
\end{equation*}
and \eqref{eqclaim4} follows.

Finally, recalling \eqref{eqlaplace p}, we estimate the pressure by considering the following elliptic equation
\begin{equation*}
\begin{cases}
-\Delta p=\partial_iv^j\partial_jv^i, & \text{ in } \Omega_t,\\
p=\mc, &  \text{ on } \parOmega_t.
\end{cases}
\end{equation*}

By standard elliptic estimates, and using the curvature bound in \eqref{eqCCC}, and \eqref{eqvH1}, we obtain
\begin{align} 
\norm{p}_{H^2(\Omega_t)}\le{}& C\Paren{ \norm{\Delta p}_{L^2(\Omega_t)}+\norm{\mc}_{H^{\frac 32}(\parOmega_t)}}\nonumber\\
\le{}& C\Paren{\norm{\partial_iv^j\partial_jv^i}_{L^2(\Omega_t)}+C(\CC)}\nonumber\\
\le{}& C\Paren{\norm{\nabla v}_{L^2(\Omega_t)}\norm{\nabla v}_{L^\infty(\Omega_t)}+C(\CC)}\nonumber\\
\le{}&C(\CC)\Paren{\norm{\nabla v}_{L^\infty(\Omega_t)}+1}.\label{eqpressureH2}
\end{align} 
Combining the above estimate with \eqref{eqmidstep} and \eqref{eqclaim4}, \eqref{eqenergy estimate} follows, and we have completed the proof of the proposition. 
\end{proof}

In the following proposition, we establish the equivalence of two energy functionals.
\begin{proposition}\label{prclose energy}
Under  \eqref{eqCCC}, for any time  $t\in (0,\Tstar)$, we have 
\begin{equation}\label{eqclose energy}
E(t) \le C(\CC)(1+\energy(t)).
\end{equation}
\end{proposition} 
\begin{proof}
Recalling that the $L^2$-norm of the vorticity is bounded by $C(\CC)$ as stated in \eqref{eqCCC}, by interpolation, we obtain
\begin{align}
\norm{\vorticity v}_{H^{2}(\Omega_t)}^2&\le C\Paren{\norm{\vorticity v}_{L^{2}(\Omega_t)}^2+\norm{\nabla^2\Paren{\vorticity v}}_{L^{2}(\Omega_t)}^2}\nonumber\\
&\le C(\CC)\Paren{\norm{\nabla^2\Paren{\vorticity v}}_{L^{2}(\Omega_t)}^2+1}\nonumber\\
&\le C(\CC)\Paren{\energy(t)+1}.\label{eqcurlH2}
\end{align} 
We aim to control
\begin{equation*}
\norm{\DT v\cdot \normal}_{L^2(\parOmega_t)}^2,\norm{\DT v}_{H^{\frac{3}{2}}(\Omega_t)}^2,\ \text{and}\ \norm{v}_{H^3(\Omega_t)}.
\end{equation*}
Using the divergence theorem, \eqref{eqCCC} and \eqref{eqextendnormal}, we have
\begin{align}
\norm{\DT v\cdot \normal}_{L^2(\parOmega_t)}^2
={}&\int_{\parOmega_t} [(\DT v\cdot \normal)\DT v]\cdot \normal dS\nonumber\\
\le{}&\int_{\Omega_t} \module{(\DT v\cdot \normal)\divergence\DT v}dx+\int_{\Omega_t} \module{\nabla\DT v\tencon\DT v}dx
+\int_{\Omega_t}\module{\DT v\tencon\nabla\normal\tencon\DT v}dx\nonumber\\
\le{}& C(\CC)\Paren{\norm{\DT v}_{L^{2}(\Omega_t)}^2+\norm{\divergence\DT v}_{L^{2}(\Omega_t)}^2 +\norm{\nabla\DT v}_{L^{2}(\Omega_t)}\norm{\DT v}_{L^{2}(\Omega_t)}}\nonumber\\
\le{}& \varepsilon \norm{\DT v}_{H^{\frac32}(\Omega_t)}^2+C_\varepsilon(\CC)\norm{\DT v}_{L^2(\Omega_t)}^2\nonumber\\
\le{}& \varepsilon \Paren{\norm{\DT v}_{H^{\frac32}(\Omega_t)}^2+\norm{v}_{H^3(\Omega_t)}^2}+C_\varepsilon(\CC).\label{eqDtvnormalL2}
\end{align}
In the last step, we have utilized \eqref{eqEuler1}, \eqref{eqpressureH2}, the interpolation inequality, and the $L^2$-bound of the velocity given in \eqref{eqCCC} to obtain
\begin{equation*} 
\norm{\DT v}_{L^2(\Omega_t)}^2\le \varepsilon\norm{v}_{H^3(\Omega_t)}^2+C_\varepsilon\norm{v}_{L^2(\Omega_t)}^2\le \varepsilon\norm{v}_{H^3(\Omega_t)}^2+C_\varepsilon(\CC).
\end{equation*}

Noting that $\vorticity\DT v$ vanishes, applying \eqref{eqDtvnormalL2} and  \eqref{eqdivcurlk} for $k=\frac 32$, it follows that
\begin{align*}
\norm{\DT v}_{H^{\frac{3}{2}}(\Omega_t)}^2\le{}& C(\CC)\Paren{\norm{\DT v\cdot\normal}_{H^{1}(\parOmega_t)}^2+\norm{\DT v}_{L^{2}(\Omega_t)}^2 +\norm{\divergence \DT v}_{H^{\frac 12}(\Omega_t)}^2}\\
\le{}& \varepsilon \Paren{\norm{\DT v}_{H^{\frac32}(\Omega_t)}^2+\norm{v}_{H^3(\Omega_t)}^2}+ C_\varepsilon(\CC)\Paren{\norm{\divergence \DT v}_{H^{\frac 12}(\Omega_t)}^2+\energy(t)+1}.
\end{align*} 
By \eqref{eqdivDtv} and \eqref{eqnablav2}, we are able to control 
\begin{equation*} 
\norm{\divergence \DT v}_{H^{\frac 12}(\Omega_t)}^2\le C \norm{\module{\nabla v}^2}_{H^{\frac 12}(\Omega_t)}^2\le C(\CC)\norm{v}_{H^{3}(\Omega_t)}^2.
\end{equation*}
By combining the above estimates, we obtain
\begin{equation*}
\norm{\DT v\cdot \normal}_{L^2(\parOmega_t)}^2+\norm{\DT v}_{H^{\frac{3}{2}}(\Omega_t)}^2
\le C(\CC)\Paren{\norm{v}_{H^3(\Omega_t)}^2+\energy(t)+1} .
\end{equation*}

Therefore, it suffices to bound  $\norm{v}_{H^3(\Omega_t)}^2$. By applying  \eqref{eqdivcurlk} for $k=3$, divergence-free condition \eqref{eqEuler2}, \eqref{eqCCC} and \eqref{eqcurlH2}, it follows that
\begin{align*}
\norm{v}_{H^3(\Omega_t)}^2&\le C(\CC)\Paren{\norm{v_n }_{H^{\frac{5}{2}}(\parOmega_t)}+ \norm{v}_{L^{2}(\Omega_t)}+\norm{\divergence v}_{H^{2}(\Omega_t)}+\norm{\vorticity v}_{H^{2}(\Omega_t)}}\\
&\le C(\CC)\Paren{\norm{\vorticity v}_{H^{2}(\Omega_t)}+C(\CC)}\\
&\le C(\CC)(\energy(t)+1).
\end{align*} 
We conclude that
\begin{equation*}
\norm{\DT v\cdot \normal}_{L^2(\parOmega_t)}^2+\norm{\DT v}_{H^{\frac{3}{2}}(\Omega_t)}^2+\norm{v}_{H^3(\Omega_t)}^2\le C(\CC)(\energy(t)+1),
\end{equation*}
and \eqref{eqclose energy} follows. This completes the proof.
\end{proof}

Finally, we complete the blow-up criterion stated in Theorem \ref{thmain1}.
\begin{proof}[Proof of Theorem \ref{thmain1}]
By applying \eqref{eqenergy estimate} and \eqref{eqclose energy}, it follows that 
\begin{equation}\label{eqfinalestimate} 
\frac{d\energy}{dt}\le C(\CC)\Paren{\norm{\nabla v}_{L^\infty(\Omega_t)}+1}(1+\energy (t)),\quad 0< t<\Tstar.
\end{equation}
Integrating the above, and using \eqref{eqass3}, \eqref{eqCCC}, and again \eqref{eqclose energy}, we have
\begin{equation}\label{eqcontrad}
\sup_{0\le t< \Tstar}\energy(t)\le C(\CC,\Tstar)\Paren{\energy(0)+1}\ \text{and}\ \sup_{0\le t< \Tstar}E(t)\le C(\CC,\Tstar)\Paren{\energy(0)+1}.
\end{equation}
Then, from the definition in \eqref{eqenergy2} and the fact that $\Tstar<\infty$, we have $v(\cdot,\Tstar)\in H^3(\Omega_{\Tstar})$ and $\nabla p\in H^{\frac 32}(\Omega_{\Tstar})$. By the trace theorem, it holds
\begin{equation*}
\norm{\mc}_{H^{2}(\parOmega_{\Tstar})}<\infty.
\end{equation*}
This gives $\parOmega_{\Tstar} \in H^4$ by Lemma \ref{leboundary regularity}. In other words, the solution $(v(\cdot, t), \Omega_t)$ does not develop singularities at time $\Tstar$ and can be extended for some time. This leads to a contradiction, and the proof is complete.
\end{proof}

\section{Proof of Theorem \ref{thmain2}}\label{se5}

In this section, we impose the assumption of simple connectedness
to prove Theorem \ref{thmain2}. Under the exclusion of Cases
(1)--(3) in Theorem \ref{thmain1}, the argument leading to the
first two lines of \eqref{eqCCC} remains valid independently of any
interior vorticity bound. In particular, the domains satisfy a
uniform interior and exterior ball condition and have uniformly
bounded $H^{\frac72}$-regularity. We therefore apply
\cite[Proposition 1 and Corollary 1]{Ferrari1993} to obtain the
following result.
\begin{lemma}\label{leFerrari}
Let $\{\Omega_t\}_{0\le t<\Tstar}$ be a family of bounded, simply
connected domains satisfying 
\begin{equation}\label{eqFerrariGeometry}
\inf_{0\le t<\Tstar}\radi(\Omega_t)>\CC^{-1}>0,
\quad
\sup_{0\le t<\Tstar}
\norm{\parOmega_t}_{H^{\frac72}}
\le C(\CC).
\end{equation}
Let $u(\cdot,t)\in H^3(\Omega_t)$ satisfy
\begin{subnumcases}
{\label{eqFerrari}}
\divergence u=0,& \text{ in } \Omega_t,\label{eqFerrari1} \\ 
u_n=0,& \text{ on } \parOmega_t.\label{eqFerrari2} 
\end{subnumcases}
Then, the following estimate holds
\begin{equation}\label{eqcurl}
\norm{u}_{W^{1,\infty}(\Omega_t)}\le C(\CC)\Brace{\Bracket{1+\log^{+}\Paren{\norm{\vorticity u}_{H^2(\Omega_t)}}}\norm{\vorticity u}_{L^\infty(\Omega_t)}+1},
\end{equation}
where $t\in(0,\Tstar)$ and
$\log^{+}(s)=\max\{0,\log s\}$.
\end{lemma}
\begin{proof}
By the Sobolev embedding on the two-dimensional boundary,
the uniform $H^{\frac72}$-bound in
\eqref{eqFerrariGeometry} implies that
\begin{equation*}
\sup_{0\le t<\Tstar}
\norm{\parOmega_t}_{C^{2,\alpha}}
\le C(\CC),
\end{equation*}
for every $\alpha\in(0,1/2)$. Together with the uniform interior
and exterior ball condition in \eqref{eqFerrariGeometry}, this allows
us to apply
\cite[Proposition 1 and Corollary 1]{Ferrari1993}
on each $\Omega_t$, with a constant uniform in $t$ and depending only
on $\CC$. This concludes the proof.
\end{proof}
We are now able to improve the blow-up criterion (4) given in Theorem \ref{thmain1}.
\begin{proof}[Proof of Theorem \ref{thmain2}] 
Let $\Tstar<\infty$ be the maximal existence time defined in
Theorem \ref{thmain1}. We argue by contradiction. Assume that none
of Cases (1)--(3) in Theorem \ref{thmain1} occurs. As in the proof of
Theorem \ref{thmain1} in Section \ref{se3}, assumptions
\eqref{eqass1}, \eqref{eqass2}, and \eqref{eqass4} hold.
We do not assume \eqref{eqass3}. Instead, suppose that 
\begin{equation}\label{eqass5}
\int_{0}^{\Tstar} \norm{\vorticity v}_{L^\infty(\Omega_t)}dt\le C(\CC).
\end{equation}
We shall derive
\begin{equation}\label{eqgradvL1BKM}
\int_0^{\Tstar}
\norm{\nabla v}_{L^\infty(\Omega_t)}dt\le C(\CC),
\end{equation}
which, together with the exclusion of Cases (1)--(3), contradicts
Theorem \ref{thmain1}.

Set 
\begin{equation*}
K(t)=\norm{\nabla v}_{L^\infty(\Omega_t)}.
\end{equation*}
Under the exclusion of Cases (1)--(3), the first two lines of
\eqref{eqCCC} remain valid. Moreover, the energy conservation law
\eqref{eqvL2} gives
\begin{equation*}
\sup_{0\le t<\Tstar}
\norm{v}_{L^2(\Omega_t)}
\le C(\CC).
\end{equation*}
Therefore, applying the div-curl estimate \eqref{eqdivcurlk} with
$k=3$, and using divergence-free condition \eqref{eqEuler2}, we obtain
\begin{equation}\label{eqvH3BKM}
\norm{v}_{H^3(\Omega_t)}\le
C(\CC)\Paren{\norm{v_n}_{H^{\frac52}(\parOmega_t)}+\norm{v}_{L^2(\Omega_t)}+\norm{\vorticity v}_{H^2(\Omega_t)}}\le
C(\CC)\Paren{1+\norm{\vorticity v}_{H^2(\Omega_t)}}.
\end{equation}
We next derive an $H^2$ estimate for the vorticity. Define
\begin{equation*}
Y(t)=\mathrm{e}+\norm{\vorticity v}_{L^2(\Omega_t)}^2+\norm{\nabla^2\Paren{\vorticity v}}_{L^2(\Omega_t)}^2,
\end{equation*}
where $\mathrm{e}$ is the natural constant.
Since $v_0\in H^3(\Omega_0)$ and $\Tstar<\infty$, we have
$Y(0)<\infty$. By enlarging $\CC$ if necessary, we may assume that
\begin{equation}\label{eq_Y0Ts}
Y(0)+\Tstar\le \CC.
\end{equation} 
By the uniform geometric estimates in the first two lines of
\eqref{eqCCC}, the interpolation constants on $\Omega_t$ are uniform
in time. Hence,
\begin{equation}\label{eqomegaH2Y}
\norm{\vorticity v}_{H^2(\Omega_t)}^2\le
C(\CC)\Paren{\norm{\vorticity v}_{L^2(\Omega_t)}^2+\norm{\nabla^2\Paren{\vorticity v}}_{L^2(\Omega_t)}^2}\le C(\CC)Y(t).
\end{equation}
From \eqref{eqDtcurlv} and the Reynolds transport formula in Lemma \ref{leReynolds},
\begin{equation}\label{eqomegaL2BKM}
\frac{d}{dt}
\norm{\vorticity v}_{L^2(\Omega_t)}^2
\le
CK(t)\norm{\vorticity v}_{L^2(\Omega_t)}^2.
\end{equation}
Moreover, Lemma \ref{le3} gives 
\begin{equation*}
\DT \nabla^2 \Paren{\vorticity v}=\nabla v  \tencon \nabla^2 \Paren{\vorticity v}+ \Paren{\vorticity v}\tencon\nabla^{3}v +\nabla^2v  \tencon \nabla \Paren{\vorticity v}.
\end{equation*} 
Consequently,
\begin{align*}
\frac12\frac{d}{dt}
\norm{\nabla^2\Paren{\vorticity v}}_{L^2(\Omega_t)}^2\le{}&C K(t)\norm{\nabla^2\Paren{\vorticity v}}_{L^2(\Omega_t)}^2\\
&+C\norm{\vorticity v}_{L^\infty(\Omega_t)}\norm{\nabla^3v}_{L^2(\Omega_t)}
\norm{\nabla^2\Paren{\vorticity v}}_{L^2(\Omega_t)}\\
&+C\norm{\nabla^2v}_{L^4(\Omega_t)}
\norm{\nabla\Paren{\vorticity v}}_{L^4(\Omega_t)}
\norm{\nabla^2\Paren{\vorticity v}}_{L^2(\Omega_t)}\\
\le{}&C K(t)\norm{\nabla^2\Paren{\vorticity v}}_{L^2(\Omega_t)}^2+CK(t)\norm{v}_{H^3(\Omega_t)}
\norm{\nabla^2\Paren{\vorticity v}}_{L^2(\Omega_t)},
\end{align*}
where we have used the estimate
\begin{equation*} 
\norm{\nabla\Paren{\vorticity v}}_{L^4(\Omega_t)}\le \norm{\nabla v}_{W^{1,4}(\Omega_t)}\le C\sqrt{K(t)}\sqrt{\norm{\nabla v}_{H^2(\Omega_t)}},
\end{equation*}
by applying \eqref{eqnablavW14}.
By \eqref{eqvH3BKM}, we obtain
\begin{equation*}
\frac12\frac{d}{dt}
\norm{\nabla^2\Paren{\vorticity v}}_{L^2(\Omega_t)}^2\le
C(\CC)K(t)\Paren{1+\norm{\vorticity v}_{H^2(\Omega_t)}^2}.
\end{equation*}
Combining this estimate with \eqref{eqomegaH2Y} and \eqref{eqomegaL2BKM}, we have
\begin{equation}\label{eqBKMvortenergy}
\frac{dY(t)}{dt}\le C(\CC)K(t)Y(t).
\end{equation}
Let $w(x,t)$ be the solution of the following Neumann problem
\begin{equation*} 
\begin{cases}
\Delta w=0,& \text{ in } \Omega_t, \\
\partial_\normal w=v_n,& \text{ on } \parOmega_t,
\end{cases}
\end{equation*} 
where $0<t<\Tstar$. Owing to the uniform geometric estimates in the first two lines of
\eqref{eqCCC}, we obtain the following Schauder estimate (cf. \cite[Lemma 4.3]{Luo2024}),
\begin{equation}\label{eqSchauder}
\norm{w}_{C^{2,\alpha}(\Omega_t)}\le C(\CC)\norm{v_n}_{C^{1,\alpha}(\parOmega_t)},\quad \alpha\in(0,1/2).
\end{equation} 
Moreover, the Sobolev elliptic estimate gives
\begin{equation*}
\norm{\nabla w}_{H^3(\Omega_t)}
\le
C(\CC)\norm{v_n}_{H^{\frac52}(\parOmega_t)}\le C(\CC).
\end{equation*}
Define
\begin{equation*}
u=v-\nabla w.
\end{equation*} 
Then $u\in H^3(\Omega_t)$ and by divergence-free condition \eqref{eqEuler2},
\begin{equation*}
\begin{cases}
\divergence u=\divergence v-\Delta w=0,& \text{ in } \Omega_t, \\ 
u_n=v_n-\partial_\normal w=0,& \text{ on } \parOmega_t.
\end{cases}
\end{equation*}
Thus, $u$ solves system \eqref{eqFerrari}. 
We also note that 
\begin{equation*}
\vorticity u=\vorticity v,\quad \text{ in } \Omega_t.
\end{equation*} 
Applying the logarithmic estimate \eqref{eqcurl} and
\eqref{eqSchauder}, and using the uniform
$H^{\frac52}$-bound of $v_n$, we obtain
\begin{align}
\norm{v}_{W^{1,\infty}(\Omega_t)}&\le\norm{u}_{W^{1,\infty}(\Omega_t)} +\norm{\nabla w}_{W^{1,\infty}(\Omega_t)}\nonumber\\
&\le C(\CC)\Brace{\Bracket{1+\log^{+}\Paren{\norm{\vorticity v}_{H^2(\Omega_t)}}}\norm{\vorticity v}_{L^\infty(\Omega_t)}+1}+C\norm{w}_{C^{2,\alpha}(\Omega_t)}\nonumber\\
&\le C(\CC)\Brace{\Bracket{1+\log^{+}\Paren{\norm{\vorticity v}_{H^2(\Omega_t)}}}\norm{\vorticity v}_{L^\infty(\Omega_t)}+1},
\label{eqgradvBKM}
\end{align}
where $t\in\Paren{0,\Tstar}$ and the constant $\alpha \in (0,1/2)$ is fixed.
Set
\begin{equation*}
L(t)=\log Y(t).
\end{equation*}
Since $Y(t)\ge \mathrm{e}$, we have $L(t)\ge1$. By
\eqref{eqomegaH2Y},
\begin{equation*}
1+\log^+\Paren{\norm{\vorticity v}_{H^2(\Omega_t)}}\le C(\CC)L(t).
\end{equation*}
It then follows from \eqref{eqBKMvortenergy} and
\eqref{eqgradvBKM} that
\begin{equation}\label{eqlogYBKM}
L^\prime(t)=\frac{Y^\prime(t)}{Y(t)}
\le C(\CC)K(t)\le C(\CC)\Brace{1+\norm{\vorticity v}_{L^\infty(\Omega_t)}L(t)}.
\end{equation}
Applying Gr\"onwall's inequality to \eqref{eqlogYBKM}, we obtain
\begin{equation*}
L(t)\le\Paren{L(0)+C(\CC)t}\exp\Brace{C(\CC)\int_0^t\norm{\vorticity v}_{L^\infty(\Omega_s)}ds}.
\end{equation*}
By \eqref{eqass5} and \eqref{eq_Y0Ts}, this yields
\begin{equation}\label{eqomegaH2uniformBKM}
\sup_{0\le t<\Tstar}\norm{\vorticity v}_{H^2(\Omega_t)}\le C(\CC).
\end{equation} 
Finally, integrating \eqref{eqgradvBKM} over
$\left(0,\Tstar\right)$ and using
\eqref{eqomegaH2uniformBKM} and \eqref{eqass5}, we obtain
\begin{equation*}
\int_0^{\Tstar}\norm{\nabla v}_{L^\infty(\Omega_t)}dt\le C(\CC). 
\end{equation*}
This proves \eqref{eqgradvL1BKM}. Hence none of the four scenarios
in Theorem \ref{thmain1} occurs, contradicting the maximality of
$\Tstar$. Therefore, assumption \eqref{eqass5} is false, and hence
\begin{equation*}
\int_0^{\Tstar}
\norm{\vorticity v}_{L^\infty(\Omega_t)}dt=\infty.
\end{equation*}
The conclusion for irrotational flows follows immediately.
This concludes the proof.
\end{proof}  
\noindent\textbf{Acknowledgment}
S. Yang was supported by the Beijing Natural Science Foundation under Grant No. 1264051, the Beijing Postdoctoral Research Activity Funding Program, and the Funding Program for Newly Recruited Young Teachers of Beijing University of Technology. C. Hao was partially supported by the National Natural Science Foundation of China (NSFC) under Grant No. 12171460, the CAS Project for Young Scientists in Basic Research under Grant No. YSBR-031 and the National Key R\&D Program of China under Grant No. 2021YFA1000800. T. Luo was supported by a General Research Fund of Research Grants Council of Hong Kong under Grant No. 11313025.

\medskip

\noindent\textbf{Data Availability} Data sharing is not applicable to this article as no datasets were generated or analyzed during the current study.

\medskip

\noindent\textbf{Declarations}
\medskip

\noindent\textbf{Conflict of interest} The authors declare that there is no conflict of interest. 

\end{document}